\documentclass{amsart}
\usepackage{graphicx}
\usepackage{amsthm,amsmath,verbatim,amssymb}
\usepackage{color}

\newcommand{\ncal}{\mathcal{N}}

\newcommand{\crit}{{\operatorname {crit}}}

\newcommand{\ga}{\gamma}
\newcommand{\De}{\Delta}

\newcommand{\szego}{Szeg\"o }

\newcommand{\inv}{^{-1}}
\newcommand{\kahler}{K\"ahler }

\newcommand{\wt}{\widetilde}
\newcommand{\wh}{\widehat}
\newcommand{\PP}{{\mathbb P}}
\newcommand{\R}{{\mathbb R}}
\newcommand{\C}{{\mathbb C}}

\newcommand{\CP}{\C\PP}
\renewcommand{\d}{\partial}

\newcommand{\dbar}{\bar\partial}
\newcommand{\ddbar}{\partial\dbar}

\newcommand{\D}{{\mathbf D}}

\newcommand{\half}{{\frac{1}{2}}}

\newcommand{\FS}{{{\operatorname{FS}}}}

\renewcommand{\phi}{\varphi}

\newcommand{\lcal}{\mathcal{L}}

\newcommand{\ocal}{\mathcal{O}}

\newcommand{\E}{{\mathbf E}\,}
\newcommand{\al}{\alpha}
\newcommand{\be}{\beta}

\newtheorem{lem}{Lemma}

\newtheorem{prop}{Proposition}
\newtheorem{thm}{Theorem}

\newtheorem{cor}{Corollary}
\newcommand{\tr}{{\operatorname{Tr}}}

\newcommand{\sn}{|s_n|_{h^n}}
\renewcommand{\d}{\partial}

\newtheorem{rem}{Remark}

\newcommand{\eh}{e^{-\frac{n\phi}2}}
 \def   \half   {{\frac{1}{2}}}
\def \to {\rightarrow}
\begin{document}\title[Critical values]{Critical values of  random analytic functions on complex manifolds}
\author[Renjie Feng]{Renjie Feng}
\author[Steve Zelditch]{Steve Zelditch}
\address{Department of Mathematics and Statistics, Mcgill University, Canada}
\email{renjie@math.mcgill.ca}
\address{Department of Mathematics, Northwestern University, USA}
\email{zelditch@math.northwestern.edu}

\thanks{Research partially supported by NSF grant  DMS-1206527.}

\date{\today}

\begin{abstract}  We study  the asymptotic  distribution of critical values
of  random holomorphic sections $s_n \in H^0(M^m, L^n)$ of powers of a positive line bundle $(L, h) \to (M, \omega)$
on a general K\"ahler manifold of dimension $m$. By critical value is meant the value of $|s(z)|_{h^n}$ at a critical point where $\nabla_h s_n(z) = 0$,
where $\nabla_h$ is the Chern connection. The distribution of critical values of $s_n$ is its empirical measure. Two main ensembles are considered: (i) the normalized Gaussian ensembles
so that $\E ||s_n||^2_{L^2} = 1$ and (ii)  the spherical ensemble  defined by Haar measure on
the unit sphere $SH^0(M, L^n) \subset H^0(M, L^n)$  with $||s_n||^2_{L^2} = 1$.
The main result is that
the expected distributions of critical values in both the normalized Gaussian ensemble and the spherical ensemble tend
to the same universal limit as $n \to \infty$, given explicitly as an integral over $m \times m$ symmetric matrices.

\end{abstract}

\maketitle
\section{Introduction}

The purpose of this article is to determine the asymptotic distribution of critical values of  random holomorphic polynomials
of large degree,  and their generalizations (`holomorphic sections')  to all compact K\"ahler manifolds $(M, \omega)$
of any dimension $m$.  We work in the same general  setting
as the articles \cite{DSZ1,DSZ2,DSZ3} on the distribution of critical points on K\"ahler manifolds and recall
the definitions in \S \ref{BACKGROUND}.  We first consider the critical value distribution of Gaussian random `polynomials'
$s_n$ as the degree $n \to \infty$ and then consider the more difficult and interesting problem of critical values of $L^2$ normalized  random polynomials with $||s_n||_{L^2} =1$. We refer to the latter as the {\it spherical critical value distribution} since the `polynomials' are drawn
at random from the unit sphere in Hilbert space.   We regard  the spherical distribution as primary for the critical
value distribution since   a critical value is only counted once
in a line $\{c s_n, c \in \C\}$ of sections and one can relate the heights of the critical values to the other
threshold heights of $L^2$ normalized `polynomials'. Theorem \ref{maintheorem}  shows that a special {\it normalized  Gaussian critical value distribution} has a universal
limit independent of the manifold and K\"ahler metric. The main result of this article,  Theorem \ref{PBLIKE},  shows that  the {\it spherical  critical value distribution}  has
the same universal  limit. We also give a limit formula for the simpler spherical  value distribution in Theorem \ref{SDV}.
The spherical limit results may be viewed as  a  Poincar\'e-Borel theorems for critical values.

The spherical  distribution of critical values is potentially useful in analyzing the Morse theory of  the modulus $|s_n(z)|_{h^n}^2$ of random sections, since it is at critical values
that the level sets change topology. They are also important in understanding the topography of random surfaces,
i.e. the graphs of the modulus  $y = |s_n(z)|^2_{h^n}$ of `polynomials' of degree $n$, which are often visualized as mountain landscapes
above some given sea level.   It is known that $\sup_{z \in M } ||s_n(z)||_{h^n} \leq C m^{n/2} $ when $||s_n||_{L^2} = 1$ and $\dim_{\C} M = m$   and it is proved in  \cite{SZ} (after a long
history of similar results in other settings)
that the expected sup norm of such normalized sections is bounded by a universal constant times $\sqrt{\log n}$.
 Thus in  a measure sense,  typical `polynomials' of degree $n$ and norm one
have global maxima  $\leq C \sqrt{\log n}$, and conjecturally the median should be of this form for some $C$
\footnote{The methods of this article give an explicit value of C, which we defer to a future article.}.  Deterministically,  critical values of all normalized
polynomials  lie in $[0, C n^{m/2}]$.
It would be interesting to know the exact  height  (for $L^2$ normalized polynomials) at which the peak of the random mountain first occurs. At a certain threshold  height,  the mountain tops are sometimes conjectured to form a Poisson
spatial process, and it would be interesting to know the  connectivity properties of the landscape at lower sea levels. The calculation of the spherical density
of critical points  is only   a  calculation but it is probably a necessary one for the more involved landacape questions.  The K\"ahler setting  is a model for other settings in which one studies normed random waves with
a notion of degree or eigenvalue, such as random spherical harmonics or more general Riemannian waves on Riemannian manifolds,
in which the principal modification is in  the asymptotics of the relevant covariance functions.

Before stating the results, we introduce some notation and background.
Compact complex manifolds have no non-constant holomorphic functions and the natural replacement for them are
twisted holomorphic functions know as holomorphic sections of  complex line bundles $\pi: L \to M$. Here, the fiber $L_z$
over $z \in M$ is a one-complex dimensional space and a holomorphic section is a map $s: M \to L$ satisfying
$\dbar s = 0, \pi \circ s(z) = z$.  Degree $n$ sections are sections of the $n$th tensor power $L^n$ of $L$, and
the space of holomorphic sections is denoted $H^0(M, L^n)$. Its dimension is given asymptotically  by
\begin{equation} \label{dn} d_n = \dim_{\C} H^0(M, L^n) \simeq C n^m. \end{equation} When $\dim_{\C} M= 1$, i.e. when $M$ is a Riemann surface, then the natural examples
are polynomials of degree $n$ ($g = 0$), theta functions of degree $b$ $(g= 1)$ and holomorphic differentials
of type $(dz)^n$ for $g \geq 2$.
The techniques and results of this article, as in the predecessors \cite{BSZ,DSZ1}, hold in this general geometric setting.

The K\"ahler metric $\omega$ determines a Hermitian metric $h$ and  connection $\nabla$ on $L$ and on its powers. The
Hermitian metric satisfies $\ddbar \log h = \omega$, and the connection $\nabla$ is  known as the Chern
connection and is compatible with the Hermitian metric $h$ on and complex structure on $L$.
As recalled in \S \ref{BACKGROUND}, the Hermitian metric $h$ and K\"ahler form $\omega$ give rise to a definition
of Gaussian random holmorphic section in $H^0(M, L^n)$.

By a critical point of
a holomorphic section $s \in H^0(M, L^n)$ we mean a point $z \in M$ so that
\begin{equation} \label{P} \nabla s(z) = 0. \end{equation}
Thus the section is `parallel' at $z$.   Equivalently, critical points are points where the norm square is critical  $d |s_n|^2_{h^n}=0$ and so we are studying
the critical points and values of the real-valued function $|s(z)|^2_{h^n}$. More precisely,
\begin{equation}  d |s_n|^2_{h^n}=0 \Longleftrightarrow \nabla  s_n=0 \,\,\,\mbox{or}\,\,\, s_n=0.\end{equation}
We note that $x = 0$ is surely a critical value of $|s_n(z)|^2_{h^n}$ since every section has zeros. However
we omit the zero value in the definition of the empirical measure of critical values.
 When the Hermitian line bundle $(L,h)$ is positive, as we assume, the only local minima of $|s|_h$ are its zeros.
Therefore the critical points in \eqref{criticalmeasure} are either saddle points or local maxima.

To define Gaussian random holomorphic sections,  introduce a family of Gaussian measures adapted to the Hermitian metric and the associated inner product
\eqref{innera} on sections. For any $\alpha >0$ we put
\begin{equation} \label{HGMg} d\gamma_{\alpha}^n(s_n)=(\frac{\alpha}{\pi})^{d_n}e^
{-\alpha|a|^2}da\,,\qquad s_n=\sum_{j=1}^{d_n}a^n_js^n_j\,,\end{equation}
where $\{s^n_1, \cdots, s^n_{d_n}\}$ is the orthonormal basis of  $H^0(M,L^n)$ with respect to the inner product \eqref{innera}.
 Equivalently, the
coefficients $a_j^n$ are complex Gaussian random variables which satisfy the following normalization conditions,
 \begin{equation}\label{law}\E_n^\alpha a^n_k=0,\,\,\,\E_n^{\alpha} a^n_k\bar a^n_j=\frac 1{\alpha}\delta_{kj},\,\,\, \E_n^{\alpha} a^n_ka^n_j=0\end{equation}
  Here, we denote  the expectation with respect to $\gamma^n_{\alpha}$ by  $\E_n^{\alpha}. $
 Under this normalization, we have the expected $L^2$ norm of $s_n$,  \begin{equation}\label{norm1}\E_n^{\alpha} \|s_n\|^2_{h^n}= \frac{d_n}{\alpha}.\end{equation}
Thus  the covariance kernel of $\gamma_{\alpha}^n$ is
\begin{equation} \label{SZ} \E_n^{\alpha}( s(z) \overline{s(w)} ) = \frac{1}{\alpha} \Pi_n(z,w),  \end{equation}
where $\Pi_n$ is the Szeg\"o projector with respect to \eqref{innera}.

 When $\alpha = d_n$ we call the
ensemble the {\it normalized Gaussian measure}. It is given by
\begin{equation} \label{HGM} d\mu_h^n(s_n)=(\frac{d_n}{\pi})^{d_n}e^
{-d_n|a|^2}da, \end{equation} and from \eqref{norm1} it follows that
\begin{equation} \label{dn1} \E_n^{d_n} ||s_n||^2_{h^n} = 1. \end{equation}
 As will be seen below, the density of critical values has a limit for this sequence of probability measures.
We discuss the relations of the densities as $\alpha$ varies in \S \ref{GAUSS}.

The distribution of critical points of a section is defined by the un-normalized  empirical measure
\begin{equation} C_s=   \sum_{z: \nabla s(z) = 0} \delta_z. \end{equation} The  Chern
connection $\nabla$ associated with $h$ is not holomorphic, and the number of
critical points depends on the section $s$.  The statistics
of the number of critical points in the normalized  Gaussian ensemble was
 determined
in \cite{DSZ1, DSZ2}. The critical point distribution is invariant under $s \to c s$ and therefore
it is equivalent to work with Gaussian or spherical distributions.  It is proved in Corollary 5 of \cite{DSZ2} that the expected number $\ncal^\crit_{n,q,h}$ of
critical points of Morse index $q$ for any positively curved Hermitian metric $h$ on a K\"ahler manifold
of dimension $m$   satisfies
\begin{equation} \label{number} \ncal^\crit_{n,q,h} \sim
\left[\frac{\pi^m b_{0q}}{m!}\,c_1(L)^m\right]n^m. \end{equation}
Here, $b_{0q}$ is a Betti number and $c_1(L)$ is the first Chern class, both of which are topological invariants.
For instance,
  in dimension one there are roughly $\frac{4}{3} n$ saddles points and $\frac{1}{3} n$ local maximal for
Gaussian random polynomials of degree $n$ with the $SU(2)$ (or Fubini-Study) inner product.

\subsection{Statement of results}

In this article we study the Gaussian, resp. spherical,  distribution of  {\it critical values},
\begin{equation}  \mathcal {CV}_s : = \label{Nabla}\{ |s(z)|_{h^n} : \nabla s(z) = 0, \;\; s \in H^0(M, L^n)\}
 \end{equation}
in the limit as the degree $n$ tends to $\infty$. Thus, the ``value" of the section at a critical point is the Hermitian norm in $\R$
of the section; since $s(z) \in L_z$ it would not make sense to study the values themselves.
  In view of \eqref{number}, we  define the (normalized) empirical measure of nonvanishing critical values of $\sn$ by
\begin{equation}\label{criticalmeasure}CV_s=\frac 1{n^m}\sum_{z:\,\, \nabla's_n(z) = 0}
\delta_{|s_n|_{h^n}}.
\end{equation}
Note that it is not necessarily a probability measure but from the results of \cite{DSZ2} (such as \eqref{number}) it
follows  that for any $\epsilon > 0$ there exists a constant $C$ so that $\# \{z: \nabla s(z) = 0\} \leq C n^m$
except for a set of sections of measure $< \epsilon$.

We define  the Gaussian density of critical values $\D_n^{\alpha}(x)$ as the expected density of $\E_n^{\alpha}\, CV_s$ in the sense of distribution,
\begin{equation}\label{expectdensity}\E_n^{\alpha}\left(\frac 1{n^m}\sum_{z:\,\,\nabla's_n=0} \psi(\sn)\right)=\int_{\R^+}\psi(x) \D_n^{\alpha}(x) dx\,\,\,\mbox{for}\,\,\psi\in C_0(\R_+)\end{equation}
 where $dx$ is the Lebesgue measure on $\R$.  In \S \ref{KR} we will
calculate the densities for all $\alpha$. The Kac-Rice  are quite complicated for fixed $n$, but for the normalized  Gaussian
ensemble there are simple asymptotics.

To state the result we need some notations.  We denote by $S(\C^m)\cong \C^\frac{m^2+m}{2}$ the space  of complex symmetric
matrices.  We also denote
by $d\xi$  the Legesgue measure on $S(\C^m)$. We also define the special matrix   $P:=(\delta_{jj'}\delta_{qq'}+\delta_{jq'}\delta_{qj'})_{\frac{m^2+m}{2}\times \frac{m^2+m}{2}}$.

\begin{thm}\label{maintheorem}Let $(M,\omega)$ be an $m$-dimensional compact \kahler manifold. Let $(L,h)\to M$  be a polarized positive holomorphic line bundle.
 Let $\D_n^{d_n}$ be the expected density of critical values defined in (\ref{expectdensity}) with $\alpha = d_n$, i.e.  the normalized
Gaussian density.
Then we have,
\begin{equation}\label{main}\D_n^{d_n}(x)=\D_{\infty}(x) + O\left(\frac{1}{n} (x( 1 + x^4) e^{-x^2})\right) \;\; \mbox{ on }\;\; (0, \infty), \end{equation}
where  $$\D_{\infty}(x) =  f_m(|x|)|x|e^{-|x|^2}, \;\; \mbox{with}\;\;f_m(|x|)=c_m\int_{S(\C^m)} e^{-|\xi|^2}\left||\sqrt P \xi|^2-|x|^2\right|d\xi. $$ Here, $c_m=\frac {2\pi}{\pi^{\frac {(m+1)(m+2)}2} }V$,
  where $V$ is the volume of $(M, \omega)$.    The asymptotics can be differentiated any
number of times (with appropriate changes in the polynomial growth  in the remainder estimate.)

\end{thm}

\begin{rem} Henceforth we generally set $V = 1$ for notational simplicity.\end{rem}

In the case of Riemann surfaces when $m=1$, we have $P=2$. Assuming the  volume of $M$ is $\pi$, then
\begin{equation} \label{n=1} \D_n^{d_n}(x)= x\left(2x^2-4+8e^{-\frac{x^2}2}\right) e^{-x^2}+O(\frac{e^{-x^2} (1 + x^5)}{n}). \end{equation}
Below is the computer graphic of the leading term,
\begin{center} \label{graph of  $f_1(x)e^{-x^2}$}\text{Graph of  $D_{\infty}(x)$ in dimension one}
\includegraphics[scale=0.5]{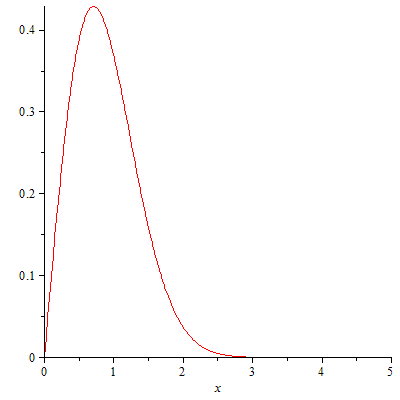}
 \end{center}
\bigskip

The critical point densities for $\D_n^{\alpha}$ have the scaling relation $\D_n^{\alpha}(x) =  \alpha^{\half} \D_n^1(\alpha^{\half} x) $ (Lemma \ref{alpha}), and from this one can determine the asymptotics for the
other Gaussian ensembles.
We also give a similar formula for the simpler  expected distribution function  of random sections in \S \ref{VDSec}.

The proof of  Theorem \ref{maintheorem} is
is  based on the Kac-Rice formula in Lemmas \ref{KR} and \ref{pz}, which give exact formulae for all of the $\D_n^{\alpha}(x)$.
The asymptotics then follow from the complete asymptotic  expansion of the covariance kernel \eqref{SZ} in \S \ref{covariance}. In the case
of $SU(m+1)$ polynomials on $\CP^{m}$ we give an exact formula for all $n$ in \S \ref{FS}.

\begin{rem}

  As in \cite{DSZ2} Theorem 1.2 (see also \cite{Bau}),
we can give  similar formulae for the distribution of critical values when the critical point is constrained to have
a specified Morse index. The formula only changes in that we integrate over the subset  $S^q(\C^m)$ of  matrices of index $q$.
\end{rem}

\subsection{Density of critical values in the spherical ensemble}

As mentioned above, the critical point distribution is homogeneous, i.e the same for $s $ and $c s$. The critical value
distribution is however not homogeneous, since the critical values are multiplied by $|c|$. Since the mass
of the normalized  Gaussian measure is asymptotically  concentrated near the unit sphere as $d_n \to \infty$, the critical values of a line of sections $\{c s\}$  are weighted most
for values of $|c| $ close to 1. This weighting is a re-scaled version of   the one  in the classical  Poincar\'e Borel theorem,
which states that the spherical probability measure $\nu_d$ on the sphere $S^d(\sqrt{d})$
tends to the Gaussian measure  as $d \to \infty$.  More precisely, if $P_d:\R^d\to \R^k$ is
$P_d(x)=\sqrt{d}(x_1,\dots,x_k)$, then for all $k$,  $P_{d*}\nu_d \to \ga_k = (2 \pi)^{-d/2} e^{-|x|^2/2} dx\,.$
Moreover,
$$\gamma_d\{x \in \R^d: ||x||^2 \geq \frac{d}{1 - \epsilon}  \} \leq e^{- \epsilon^2 d/4},\; \;\; \gamma_d\{x \in \R^d: ||x||^2 \leq (1 - \epsilon) d \} \leq e^{- \epsilon^2 d/4}. $$

Our  spherical probability measure is normalized  Haar measure  $d\nu_n $ on \begin{equation} SH^0(M, L^n) =\{s \in H^0(M, L^n) : ||s||_{L^2} = 1\}.
\end{equation}  We refer to the corresponding
probability space as the  spherical ensemble. What we are calling the normalized Gaussian measure \eqref{HGM}
concentrates exponentially on this unit sphere.
We denote the expectation with respect to $d \nu_n$ by $\E_{\nu_n}$ and define the (normalized) spherical density of critical points $\D_n^S(x)$ by
\begin{equation} \label{SPHDEN} \E_{\nu_n} CV_s = \D^{S}_n(x) dx.  \end{equation}
In fact it makes more sense to pass to the quotient   Fubini-Study probability measure  on the projective space $\PP H^0(M, L^N)$ of sections since
the critical value distribution is invariant under multiplication by $e^{i \theta}$.

We view the spherical density \eqref{SPHDEN} as primary   because fixing $||s_n||_{L^2} = 1$ sets
a scale against which one can calibrate the  heights at which interesting features of the landscapes occur.
For instance, as mentioned above,   the  spherical critical value distribution  $\D_n^S$  is supported in $[0,  C \sqrt{d_n}]$, and its median should occur at a constant times $ \sqrt{\log n}$.  In the sequel we plan to study such distinguished levels
in more detail.
The main result of this article is:

\begin{thm} \label{PBLIKE}  The density of critical values in the spherical ensemble $SH^0(M, L^n)$ on
any compact K\"ahler manifold has the universal limit,
$$\lim_{n \to \infty} \D_n^S (u) = D_{\infty}(u). $$
\end{thm}
Thus, the spherical critical point
distribution tends to the same universal limit as the normalized  Gaussian measure  of Theorem \ref{maintheorem}.
As the asymptotics (and Graph) indicate, the most probable critical value and the  median of the  critical value distribution
$\D_n^S$  is around $1$ in dimension one. An  upper bound for the median may be derived from the exact formula
for the spherical critical point density. In a subsequent  article we will apply Theorem \ref{PBLIKE}  or more precisely
its proof to  obtain a formula for the median and for the asymptotics of the critical point density with fixed Morse
index in special $n$-dependent intervals.

As discussed above, Theorem \ref{PBLIKE} is a Poincar\'e-Borel type theorem for critical values. Intuitively it
is based on the concentration of normalized  Gaussian measure around the unit sphere, but in its details it uses the special
scalings of the critical value distribution and the asymptotics of the covariance kernels and therefore does not
seem to follow directly from the classical Poincar\'e-Borel theorem. The proof is based on a Laplace transform
relation between  the spherical and normalized  Gaussian critical value distributions.

The Poincar\'e-Borel relation between the normalized  Gaussian and spherical expectations
of the critical value distribution holds also for the full value distribution. We denote the density
of values in the spherical ensemble by $f_n^S(u) du$. In
 \S \ref{VDSec} we prove:

\begin{thm} \label{SDV}  The density of values in the  spherical ensemble $SH^0(M, L^n)$ of any K\"ahler manifold
has the limit, $$\lim_{n \to \infty} f^S_n(u) =  2 u e^{- u^2}. $$ \end{thm}

\subsection{Related results and problems}

Other articles over the last ten years  devoted to the statistics of critical points
of Gaussian random fields   include  \cite{DSZ1,DSZ2,DSZ3,Bau, Mac,NSV, B,GW,ABA,N1,N2,Z}. The Kac-Rice
formula for the Gaussian  critical value distribution of holomorphic sections were originally obtained in \cite{SZ3} but the asymptotics were
not determined as explicitly as in this article. As mentioned above, we view the Gaussian formalism mainly as a
method for computing
 the spherical distribution.   In the real domain, a Kac-Rice formula and an
asymptotic analysis of the critical point distribution are given in \cite{N2}.

The expected value of the empirical measure  \eqref{criticalmeasure}  is only the first and simplest of the many
probabilistic problems on critical values. As in the case of zeros or critical points, one may ask for the variance
of  linear statistics (pairings of smooth test functions with the empirical measure), the asymptotic normality of
linear statistics, large deviations properties and so on. As mentioned above, the precise structure of the landscapes
defined by $y = |s_n(z)|_{h^n}$ is unknown in many respects.

In the case of polynomials of one complex variable $p(z)$, one might instead use
the standard complex derivative $\frac{dp}{dz} $ to define critical points and critical values, but it is in fact a meromorphic connection with a pole at infinity
and leads to quite a different theory. To our knowledge, statistics of critical points in the latter sense have only been studied in \cite{H,FW}.
In \cite{FW}, the authors studied the expected density of critical values of Gaussian $SU(2)$ random polynomials $p_n$ defined on $\C$ with respect to the meromorphic connection $\frac{d}{dz}$.  Also, in \cite{H}, B. Hanin studies
the correlation between zeros and critical points for this classsical connection.

 The main result of \cite{FW} is that the (un-normalized) expected density of nonvanishing critical values of the modulus of $|p_n|$ satisfies $$\E\left(\sum_{z:\,\,|p_n|'=0}\delta_{|p_n|}\right)\sim 1/x$$ on $\R_+$ as $x$ away from $0$ where $|p_n|'=\frac{\d |p_n|}{\d z}$ .
This result is quite different from Theorem \ref{maintheorem}, due to the fact that the connection
 $d/dz$ is a flat meromorphic one rather than   the smooth but non-holomorphic connection of this article.

\section{\label{BACKGROUND} Background on Gaussian measures in the K\"ahler setting}

\subsection{\kahler manifolds}
Let $(M,\omega)$ be a compact \kahler manifold of complex dimension $m$ with the local \kahler potential $\omega=\d\dbar\phi$. Let $(L,h)\to M$ be a positive holomorphic line bundle such that the curvature of the Hermitian metric $h$,
\begin{equation}\Theta_h=-\d\dbar \log |e|_h^2\end{equation}
is a positive $(1,1)$ form \cite{GH}. Here, $e$ is a local non-vanishing holomorphic section of $L$ over an
open set $U \subset M$ such that locally $L|_U\cong U\times\C$ and $|e|_{h} = h(e, e)^{
1/2}$
is the pointwise $h$-norm of $e$.

We denote by $H^0(M,L^n)$ the space of global holomorphic sections of $L^n=L \otimes \cdots \otimes L$. Under the local coordinate, we can write the global holomorphic section as $s_n=f_ne^{\otimes n}$ where $f_n$ is a holomorphic function on $U$.
We denote the dimension of $H^0(M,L^n)$ by $d_n$.

The Hermitian
metric $h$ induces a Hermitian metric $h
^n$ on $L^n$ given by $|e^{ \otimes n}|_{h^n}=|e|_h^n$. Throughout the article, we assume the polarized condition $\Theta_h=\omega$ such that in the local coordinate, we have the $h$-norm $|e|_h=e^{-\frac \phi 2}$ and hence $|s_n|_{h^n}=|f_n|e^{-\frac {n\phi} 2}$.

We decompose the Chern connection  $\nabla=\nabla'+\nabla''$ of the Hermitian line bundle $(L^n,h^n)$  into holomorphic and antiholomorphic parts where in the local coordinate $\nabla'=d_z+n\d\log h$ and $\nabla''=d_{\bar z}$ \cite{GH}.

We can define an inner product on $H^0(M,L^n)$ as,
\begin{equation}\label{innera}\langle s^n_1, s^n_2\rangle_{h^n}=\int_M h^n(s^n_1, s^n_2)dV\end{equation}
 where $dV=\frac{\omega^m}{m!}$ is the volume form. We recall that  throughout the article we assume the volume is normalized as $\int_M dV=1$.

We write  this in the local coordinates,

\begin{equation}\label{innera2}\langle s^n_1, s^n_2\rangle_{h^n}=\int_Mf_1\bar f_2 e^{-n\phi}dV\end{equation}
where
$s_1^n=f_1^ne^{\otimes n}$ and $s_2^n=f_2^ne^{\otimes n}$.

\subsection{Gaussian measures}

Recall that a Gaussian measure on $\R^n$ is a measure of the form
$$\ga_\De = \frac{e^{-\half\langle\De\inv
x,x\rangle}}{(2\pi)^{n/2}\sqrt{\det\De}}  dx_1\cdots dx_n\,,$$ where $\De$
is a positive definite symmetric $n\times n$ matrix.  The matrix $\De$ gives
the second moments of $\ga_\De$:
\begin{equation}\label{moments}\langle x_jx_k
\rangle_{\ga_\De}=\De_{jk}\,.\end{equation} This
Gaussian measure is
also characterized by its Fourier transform
\begin{equation}\label{muhat}\wh{\ga_\De}(t_1,\dots,t_n) = e^{-\half\sum
\De_{jk}t_jt_k}\,.\end{equation}
If we let $\De$ be the $n\times n$ identity matrix, we obtain the standard
Gaussian measure on $\R^n$,
$$\ga_n:=\frac{1}{(2\pi)^{n/2}}e^{-\half |x|^2} dx_1\cdots dx_n\,,$$
with the
property that the
$x_j$ are independent Gaussian variables with mean 0 and variance 1.

 A complex  Gaussian measure on $\C^k$ is a measure of the
form
\begin{equation}\label{cxgauss}\ga_\De= \frac{e^{-\langle\De\inv
z,\bar z\rangle}}{\pi^{k}{\det\De}} dz\,,\end{equation}
where $dz$ denotes Lebesgue measure on $\C^k$, and $\De$ is a
positive definite Hermitian $k\times k$ matrix.  The  matrix
$\De=\big(\De_{\al\be}\big)$ is the covariance matrix of $\ga_\De$:
\begin{equation}\label{covar}
\langle
z_\al\bar z_\be\rangle_{\ga_\De} =\De_{\al\be},\quad\quad 1\leq \al,\be\leq
k\;.\end{equation}

\section{\label{GAUSS} A one-parameter family of Gaussian measures and their critical point distributions}

The one-parameter family of complex Gaussian measures \eqref{HGMg}  on $H^0(M, L^n)$ may be written formally as
$$d\gamma^n_{\alpha} =(\frac{ \alpha}{\pi}) ^{d_n} e^{- \alpha ||s||^2} D s $$ where $Ds$ is Lebesgue measure.

The normalization comes
from the calculation
\begin{equation} \label{GSPH} \begin{array}{lll} 1 =\frac{ \alpha^{d_n}}{\pi^{d_n}} \int_{\C^{d_n}} e^{- \alpha ||z||^2} d z &= & \frac{1}{\pi^{d_n}}
  \int_{\C^{d_n}} e^{-  ||z||^2} d z \\&&\\&&=  \frac{\omega_{2d_n}}{2\pi^{d_n}} \int_0^{\infty} e^{- \rho} \rho^{d_n -1} d \rho'
=  \frac{\omega_{2d_n}}{2\pi^{d_n}}  \Gamma(d_n). \end{array}  \end{equation}
where $\omega_k= \frac{2\pi^{k/2}}{\Gamma(k/2)}$ is the surface measure of the unit sphere $S^{k-1} \subset \R^{k}$.

 As in \eqref{expectdensity}
we denote the normalized density of critical values with respect to $\gamma_n^{\alpha}$ by
\begin{equation} \label{CVALPHA} \E_n^{\alpha} CV_s \;  = \; \D_n^{\alpha}(x) dx. \end{equation}

We note that $\D_n^{\alpha} dx$ is not a  probability measure on $\R_+$ since $CV_s$ is not in general
a probability measure. The mass of the measure can be determined from the eventual Kac-Rice formulae
of \S \ref{KR}.

\subsection{Normalized Hemitian Gaussian measure}

 As mentioned in the introduction, when $\alpha = d_n$ it's the normalized Gaussian measure \eqref{HGM} which
is characterized by
 \begin{equation}\label{standard} \E  a^n_k=0,\,\,\,\E a^n_k\bar a^n_j=\frac 1{d_n}\delta_{kj},\,\,\, \E a^n_ka^n_j=0\end{equation}
 Under this normalization,  the expected $L^2$ norm of $s_n$ is $1$ \eqref{dn}.

\subsection{Relations between Gaussian critical value densities}

Next we compare densities \eqref{CVALPHA}  as $\alpha$ changes.
The first step is

\begin{lem}  \label{r} For any $s \in H^0(M, L^n)$ with non-degenerate critical points, and for any $r > 0$,
$$CV_{r s} = CV_s(r^{-1} x). $$

\end{lem}

\begin{proof}  For any $f \in C(\R)$ we have,
$$\langle CV_{r s} , f \rangle = \sum_{z: \nabla s(z) = 0} f(r |s(z)|_{h^n}) = \langle CV_s(x), f(r x) \rangle. $$
Changing variables to $y = rx$ completes the proof.

\end{proof}

Since $\D_n^{\alpha} dx$ transforms by the inverse dilation, the density has the transformation law,

\begin{lem}\label{alpha}
$\D_n^{\alpha}(x) =  \alpha^{\half} \D_n^1(\alpha^{\half} x) $.
\end{lem}

\begin{proof} We have,
$$\begin{array}{lll} \D_n^{\alpha}(x) dx & = & \E^n_{\alpha} CV_s = (\frac{\alpha}{\pi})^{d_n} \int_{H^0(M, L^n)}  CV_s(x) e^{- \alpha ||s||^2} D s\\&&\\
& = &  (\frac{\alpha}{\pi})^{d_n}  \int_{H^0(M, L^n)}  CV_s(x) e^{-  ||\alpha^{\half} s||^2} D s
\\&&\\
& = &  \frac{1}{\pi^{d_n}} \int_{H^0(M, L^n)}  CV_{\alpha^{-\half} s'} (x) e^{-  ||s'||^2} D s'\\&&\\
& = &  \frac{1}{\pi^{d_n}} \int_{H^0(M, L^n)}  CV_{s'} (\alpha^{\half} x) e^{-  ||s'||^2} D s' = \alpha^{\half} \D_n^1(\alpha^{\half} x ) dx, \end{array} $$
where we change variables $s \to  s'= \alpha^{-\half} s$ and apply Lemma \ref{r}.
\end{proof}

Of course this implies that $\D^{\alpha}_n(x) dx$ all have the same mass.
Since $\alpha = d_n$ in the normalized Gaussian measure, we have:
\begin{cor} \label{STANDARDCV} $ \D_n^{d_n}(x) = d_n^{\half} \;  \D_n^1(d_n^{\half} x) $
\end{cor}

\section{\label{KRSECT} Kac-Rice formula for the Gaussian critical point density}

The main result of this section (Lemma \ref{KR}) gives a genric Kac-Rice formula for the expected density $\D_n^{d_n}$
of critical points with respect to the normalized Gaussian measure. Very general Kac-Rice formulae applicable to the critical point density in our setting
are given in \cite{BSZ}.  Other presentations can be found in \cite{AT,N1}. The historical references
are  \cite{K,R}.

The Kac-Rice formula is as follows, let $f(t)$ be a real valued smooth stochastic process on the finite interval $I\subset \R$.  Then
the expected number zeros, \begin{equation}\label{kr}\E\#\{t\in I: \,\,f(t)=0\}=\int_I \int_\R |y|p_t(0,y)dxdt  \end{equation}
where $p_t(0,y)$ is the joint density of $(f, f')$ evaluated at $(0,y)$, $dy$ and $dt$ are Lebesgue measures on $\R$. If $f$ is a Gaussian process, then the joint  density $p_t(x,y)$ is uniquely determined by the covariance matrix of $(f,f')$ \cite{BSZ,AT}.

\subsection{Kac-Rice formula for the critical point density}
In this subsection, we will derive the Kac-Rice formula for the expected density $\D_n^{d_n}$ of critical values of $\sn$
with respct to $\gamma_n^{d_n}$. As we will show this particular Gaussian density has a limit as $n \to \infty$.
The formula may be derived from \cite{BSZ, DSZ1} but we take advantage of some simplifications to speed up the proof.
To simplify notation, we write $$\D_n := \D_n^{d_n}. $$

In the local coordinate $U\cong\C^m$ and a local trivialization of $L$, we write the normalized Gaussian random sections as, \begin{equation}
\label{snfn} s_n=f_ne^{\otimes n}. \end{equation} where \begin{equation}\label{fn}f_n=\sum_{j=1}^{d_n}a_jf_j^n\end{equation}
 where $\{a_j\}$ are normalized  Gaussian random variable \eqref{standard} and locally $\{s_j^n=f_j^ne^{\otimes n},\,\, j=1,\cdots, d_n\}$ is an orthogonal basis of $H^0(M,L^n)$ with respect to the inner product \eqref{innera}.

 The smooth Chern connection
then has the form \cite{GH}, \begin{equation}\nabla's_n=(f'_n-n\d\phi f_n)e^{\otimes n}.\end{equation}
Here, $d = \partial + \dbar$ is the decomposition into terms of type $(1,0)$, resp. $(0,1)$. Here, $f_n' = \partial f
= \sum_j \frac{\partial f}{\partial z_j} dz_j$.
Thus, in the local coordinate, the empirical measure \eqref{criticalmeasure}
 has the form
\begin{equation}CV_s=\frac 1 {n^m}\sum_{z\in \Omega}\delta_{|f_n|e^{-\frac{n\phi}2}},\end{equation}
where
\begin{equation}\Omega=\{z\in \C^m:\,\,(f'_n-n\d\phi f_n)e^{-\frac{n\phi}2}=0\}.\end{equation}

We also introduce the locally defined empirical measure of complex
critical values
\begin{equation} \widehat{CV}_s= \frac 1{n^m}\sum_{z\in \Omega}\delta_{f_n \eh}.\end{equation}
 We will determine the  expected  density $\widehat{\D}_n$ of $\widehat{CV}_s$ on $\C$ and then integrate out the angle
variable  $\int_0^{2\pi}\widehat{\D}_n(|x|, \theta)|x|d\theta$ to obtain the (global) expected density of $CV_s$.
In other words, we use that $\langle \psi,  CV_s \rangle = \langle \psi, \widehat{ CV}_s \rangle$
for radial functions $\psi$ and thus, $CV_s = \pi_*\widehat{ CV}_s$ where $\pi: \C^* \to \R_+$
the map $z \to |z|$.

The result is:

\begin{lem} \label{KR}The expected distribution of critical values $ CV_s$ is given by the formula,
\begin{equation}\D_n(x)=\frac 1{n^m}\int_0^{2\pi} \int_{M}\int_{S(\C^m)}p^n_{z}(x,\theta,0,\xi)\left|\det\left(\xi\xi^*-n^2I x^2\right)\right|xd\xi  dV d\theta\end{equation}
where $\xi\in S(\C^m)\cong \C^{\frac{m(m+1)}2}$ is the space of $m\times m$  complex symmetric matrices and $d\xi$ is the Lebesgue measure on $S(\C^m)$,
$p^n_{z}(x,\theta,0,\xi)$ is the joint density of $p^n_{z}(y,0,\xi)$ of normalized Gaussian random variables $(f_n, f_n', f_n'')$ evaluate at $f_n'=0$, here we substitute $y:=xe^{i\theta}$ by the map $\pi$. (The formula of $p^n_z(y,0,\xi)$
is given explicitly in Lemma \ref{pz}).

\end{lem}

\begin{proof}

We first introduce some notations: $$p_n=f_n\eh \in\C,\,\,\, q_n=(f_n'-n\d\phi f_n)\eh\in \C^m,\,\,\,r_n=f_n''\eh\in S(\C),$$ then
$p_n$, $q_n$ and $r_n$ are all complex Gaussian random variables.

By definition of the delta function, we have for any test functions $\psi\in C_0^\infty(\C)$,
\begin{align*}&\langle \frac 1{n^m}\sum_{z\in \Omega}\delta_{f_n\eh},\psi\rangle\\
=&\frac 1{n^m}\sum_{z:\,\,\, q_n(z)=0}\psi(p_n)\\=&\frac 1{n^m}\int_{\C^m} \delta_0(q_n)\psi(p_n) dq_n\wedge d\bar{q}_n\\=&\frac 1{n^m}
\int_{\C^m} \delta_0(q_n)\psi(p_n)\left|\det\left(|\frac{\d q_n}{\d z}|^2-|\frac{\partial q_n}{\d\bar z}|^2\right)\right|dz
\end{align*} where $dz $ is the Lebesgue measure on $\C$.
By direct computations, we have,
\begin{align*}\frac{\d q_n}{\d z}&=(f_n''-n\d\phi f_n'-n\d^2\phi f_n)\eh-\frac n 2 \d\phi q_n \\
&=f_n''\eh-n\d\phi q_n^*+n^2 \d\phi \d\phi^*p_n-n\d^2\phi p_n-\frac n 2 \d\phi q_n^*
\end{align*}
and
$$\frac{\d q_n}{\d \bar z}=-n\d\dbar \phi p_n-\frac n 2 \dbar\phi q_n^* $$
where $\frac{\d q_n}{\d z}$ and $\frac{\d q_n}{\d \bar z}$ are $m\times m$ symmetric matrices; for simplicity, for any matrix $A$ we denote, $$|A|^2:=AA^*$$

By taking expectation on both sides, we have,
\begin{align*}&\E\langle \frac 1{n^m}\sum_{z\in \Omega}\delta_{f_n\eh},\psi\rangle \\
=& \frac 1{n^m}
\E\int_{\C_z^m} \delta_0(q_n)\psi(p_n)\left|\det\left(|\frac{\d q_n}{\d z}|^2-|\frac{\partial q_n}{\d\bar z}|^2\right)\right|dz\\
=&\frac 1{n^m} \int_{\C_z^m}\int_{\C_y}\int_{S_\xi(\C^m)}  \psi(y) p_z(y,0,\xi) \left|\det\left(|\xi+n^2 \d\phi\d\phi^*y-n\d^2\phi y|^2-n^2|\d\dbar \phi|^2 |y|^2 \right)\right|
d\xi  dz  dy
\end{align*}
where $p_z(y,s,\xi)$ is the joint probability of the Gaussian random field $(p_n, q_n, r_n)$ and $dx$ is the Legesgue measure on $\C$.

Thus, the expected density $\E \widehat{CV}_s$ is,
\begin{equation}\label{d1}\widehat\D_n(y)=\frac 1{n^m} \int_{\C^m}\int_{S(\C^m)}p_z(y,0,\xi)\left|\det\left(|\xi+n^2 \d\phi\d\phi^*y-n\d^2\phi y|^2-n^2|\d\dbar \phi|^2 |y|^2 \right)\right|d\xi dz\end{equation}
We rewrite $y$ in  polar coordinate $(x,\theta)$. Then the expected density of $CV_s$ is,
\begin{equation}\label{d2}\D_n(x)=\int_0^{2\pi} \widehat \D_n(x,\theta)xd\theta\end{equation}

The next step is to get an explicit geometric formula for  $\D_n(x)$.
The empirical measure $CV_s$ and its average  $\D_n$ are independent of coordinates and frames on the \kahler manifold $M$ and line bundle $L^n$. We  may  choose  \kahler normal coordinate to simplify the above integral.

We freeze at a point $z_0$ as the origin of the coordinate patch to simplify the integrand at $z_0$. It is well known that in terms of \kahler normal coordinates $\{z_j\}$,
the K\"ahler potential $\phi$ has the expansion in the neighborhood of $z_0$: \begin{equation}\label{kahler}\phi(z,\bar z)= \|z\|^2 -\frac 14 \sum R_{j\bar
kp\bar q}(z_0)z_j\bar z_{\bar k} z_p\bar z_{\bar q} +
O(\|z\|^5)\;.\end{equation}In general, $\phi$ contains a
pluriharmonic term $f(z) +\overline{f(z)}$, but  a change of frame for $L$
eliminates that term up to fourth order.
Thus \begin{equation}\label{z0}\d\phi(z_0)=0,\, \d^2\phi(z_0)=0,\,\,\d\dbar\phi(z_0)=1,\,\, dV(z_0)=dz .\end{equation} Such frames are called adapted in \cite{BSZ, DSZ2}.
Hence, the joint density of $(p_n, q_n, r_n)$ at $z_0$ is the same as the joint density of Gaussian process $(f_n, f_n', f_n'')$.

Thus by (\ref{d1})(\ref{d2})(\ref{z0}), we obtain the global expression,
$$\D_n(x)=\frac 1{n^m}\int_0^{2\pi} \int_{M}\int_{S(\C^m)}p^n_{z}(x,\theta,0,\xi)\left|\det\left(|\xi|^2-n^2 I x^2\right)\right|xd\xi  dV d\theta$$
which completes the proof.
\end{proof}

\section{\label{JPD} Calculation of the joint probability density}
In this section we calculation the joint probability distribution $p^n_z$
of Lemma \ref{KR} with respect to the normalized Gaussian measure $\gamma_n^{d_n}$.

\subsection{Density $p_z(x,s,\xi)$}
In this subsection, we will derive the formula for $p^n_z(y,s,\xi)$ of the joint density of the Gaussian process $(f_n,f_n',f_n'')$.
It is  given by the formula \cite{BSZ,AT},
\begin{equation}p^n_z(y,s,\xi) =\frac 1{\pi^{d_m}}\frac 1{\det\Delta^n_z}\exp\left\{-\left\langle \begin{pmatrix}y\\ s \\ \xi \end{pmatrix}, (\Delta^n_z)^{-1}\begin{pmatrix}\bar y\\ \bar s \\ \bar \xi \end{pmatrix}\right \rangle\right\}
\end{equation}
where $\frac {(m+1)(m+2)}2$ is the dimension of the Gaussian process $(f_n,f_n',f_n'')$ and $\Delta^n_z$ is the covariance matrix of this process.

We rearrange the order the Gaussian process and write $\tilde\Delta^n_z$ as the covariance matrix of $(f_n', f_n'', f_n)$, then we rewrite,
\begin{equation}p_z^n(y,s,\xi) =\frac 1{\pi^{d_m}}\frac 1{\det\tilde\Delta^n_z}\exp\left\{-\left\langle \begin{pmatrix}s\\ \xi\\y \end{pmatrix}, (\tilde\Delta^n_z)^{-1}\begin{pmatrix}\bar s\\ \bar \xi  \\ \bar y\end{pmatrix}\right \rangle\right\}
\end{equation}
The  covariance kernel is defined by  \begin{equation}\label{cov}\Pi_n(z,w):=\sum_{j=1}^{d_n}f_j^n(z)\bar f^n_j(w). \end{equation} where $\{f_j^n\}$ is defined in \eqref{snfn}.

Then, we have \begin{equation}\label{wes}\E(f_n(z)\bar f_n(w))=\frac 1 {d_n}\Pi_n(z,w)\end{equation}
The notation $\Pi_n(z,w)$ usually refers to the Szeg\"o kernel but in fact \eqref{wes} is the Bergman kernel, which
has the pointwise TYZ expansion \cite{L,T,Ze},
\begin{equation} \label{bergman}\Pi_n(z, z) = n^m e^{n \phi(z)} \left[1+a_1(z) n^{-1}
+a_2(z) n^{-2}+\cdots\right]\,,\end{equation}
where $a_1$ is the scalar curvature.
Integrating over $M$ with respect to $e^{- n \phi} dV$  gives the well-known dimension polynomial,
\begin{equation}\label{rock}
d_n=n^m(1+n^{-1}\int_M a_1dV+n^{-2}\int_M a_2 dV+\cdots)\end{equation}

The covariance matrix is then given by,
\begin{equation}\label{Delta}
\tilde \Delta^n_z=\begin{pmatrix} A_n  &B_n\\ B_n^* & C_n\end{pmatrix}\;,\end{equation}
where
\begin{equation}\label{a}A_n=\frac 1 {d_n}\frac{\d^2 \Pi_n(z,w)}{\d z\d \bar w}|_{z=w}\end{equation}
is an $m\times m$ matrix;

\begin{equation}\label{b}B_n=\frac 1 {d_n}\begin{pmatrix}
 \frac{\d^3 \Pi_n(z,w)}{\d z\d^2 \bar w}|_{z=w}&\frac{\d \Pi_n(z,w)}{\d z}|_{z=w}\end{pmatrix}\end{equation}
is an $m\times \frac {m^2+m+2}{2}$ matrix;

\begin{equation}\label{c}C_n=\frac 1 {d_n}\begin{pmatrix}
\frac{\d^4 \Pi_n(z,w)}{\d^2 z\d^2\bar w}|_{z=w}&\frac{\d^2 \Pi_n(z,w)}{\d^2 z}|_{z=w}
\\\frac{\d^2 \Pi_n(z,w)}{\d^2 \bar w}|_{z=w}
&\Pi_n(z,z)
\end{pmatrix}\end{equation}
is a $\frac {m^2+m+2}{2}\times \frac {m^2+m+2}{2}$ matrix.

Thus, we have
\begin{lem} \label{pz}
With the above notations,
\begin{equation}\label{denty}p^n_z(y,0,\xi) =\frac 1{\pi^{d_m} }\frac 1{\det A_n \det \Lambda_n}\exp\left\{-\left\langle \begin{pmatrix} \xi  \\y\end{pmatrix}, \Lambda_n^{-1}\begin{pmatrix} \bar \xi \\ \bar y \end{pmatrix}\right \rangle\right\},
\end{equation}
where
\begin{equation}\Lambda_n=\frac1{d_n}(C_n-B_n^* A_n^{-1} B_n).\end{equation}

\end{lem}

\section{\label{FS} Calculation in the Fubini-Study case}

In this section, we give explicit formulae for the Kac-Rice density of critical points for $SU(m+1)$ polynomials.
 This is
the case where  $M=\CP^m$, where  $L$ is the line bundle $\ocal(1)$  whose sections are linear
functions on $\C^{m+1}$,  and so sections $L^n = \ocal(n)$  are homogeneous polynomials of degree $n$.
We equip $\ocal(1)$ with its Fubini-Study metric $h_\FS$ given
by
\begin{equation}\label{hfs} \|s\|_{h_\FS}([w])=\frac{|(s,w)|}{|w|}\;,
\quad\quad
w=(w_0,\dots,w_m)\in\C^{m+1}\;,\end{equation} for $s\in\C^{m+1*}\equiv
H^0(\C\PP^m,\ocal(1))$, where $|w|^2=\sum_{j=0}^m |w_j|^2$ and
$[w]\in\C\PP^m$
denotes the complex line through $w$. The \kahler form on $\CP^m$ is the
Fubini-Study form
\begin{equation}
\omega_\FS=\frac{\sqrt{-1}}{2}\Theta_{h_\FS}=\frac{\sqrt{-1}}{2}
\ddbar \log |w|^2 \,.\end{equation}

We denote by $\D_n^{SU(m + 1)}$ the density of critical points for $\gamma_n^{d_n}$. Our result is:
\begin{prop} \label{DSUM}
\begin{equation}\D^{SU(m+1)}_n(x)=c_mxe^{-x^2} \int_{S(\C^m)}e^{-|\xi|^2} \left|\det\left(\frac{n-1}{n}|\sqrt P\xi|^2-Ix^2\right)\right|d\xi  \end{equation}
where $c_m$ is as in Theorem \ref{maintheorem}.
\end{prop}
We note that the only difference between $\D_n^{SU(m+1)}$ and the limit $\D_{\infty}$ of Theorem \ref{maintheorem}
is in the limit $\frac{n-1}{n} \to 1$.
In dimesnion $m = 1$,
\begin{equation}\label{resultSU(2)}\D_n^{SU(2)}(x)=
\frac 4{\pi}xe^{-x^2}\int_{0}^\infty e^{-r} \left|\frac{n-1}{n}2r-I x^2\right|dr. \end{equation}

\begin{proof}

The
Szeg\"o kernel for $\ocal(n)$ is given
in an affine chart $Z_0 = 1$   with $z_j = \frac{Z_j}{Z_0}$ by
$$ \Pi_{H^0(\CP^1,\ocal(n))}(z,w) = \frac {n+1}{\pi}
(1+z \bar w)^n e^{\otimes n}(z)\otimes \overline{e^{\otimes n}(w)}\;.$$  Since our formula is invariant when
the \szego kernel is multiplied by a constant, we can replace the
above by the {\it normalized \szego kernel}  \begin{equation}
\label{NSK} \wt \Pi_n(z,w):=(1+z \bar w)^n \end{equation} in our
computation.

Since
$$\phi(z)=-\log|e(z)|^2_\FS=\log (1+|z|^2)\;,$$ we have
$$\phi(0) = \frac {\d \phi}{\d z} (0)
=\frac {\d^2\phi}{\d z^2} (0) = 0\;; $$ i.e., $e_n$ is an adapted
frame at $z=0$. Hence when computing the (normalized) matrices
$B_n,\ C_n$ for $H^0(\CP^1,\ocal(n))$, we can take the usual
derivatives of $\wt\Pi_N$. Indeed, we have
\begin{eqnarray*} \frac{\d\wt\Pi_n}{\d{z}} &=&
n(1+z \bar w)^{n-1}\bar w\;,\\
\frac{\d^2\wt\Pi_n}{\d^2{z}} & = & n(n-1)(1+z \bar w)^{n-2} (\bar w)^2\\
\frac{\d^2\wt\Pi_n}{\d^2 \bar w}& = & n(n-1)(1+z \bar w)^{n-2} z^2\\
\frac{\d^2\wt\Pi_n}{\d{z}\d\bar w} &=& n(1+z\bar w)^{n-1} +
n(n-1)(1+z\bar w)^{n-2}z\bar w\;.\\
\frac{\d^3\wt\Pi_n}{\d{z}\d^2\bar w}& = &2 n(n -1) (1+z\bar w)^{n-2} z +
n(n-1)(n-2)(1+z\bar w)^{n-3}z^2\bar w  \\
\frac{\d^4\wt\Pi_n}{\d^2{z}\d^2\bar w} &=&
2 n(n -1) (1+z\bar w)^{n-2}  +4 n(n -1) (n-2)(1+z\bar w)^{n-3} z \bar{w}\\ && +
n(n-1)(n-2)(n-3)(1+z\bar w)^{n-4}z^2\bar w^2
\end{eqnarray*}

It follows that
\begin{equation}\label{a2}A_n=\frac 1 {d_n}\frac{\d^2 \Pi_n(z,w)}{\d z\d \bar w}|_{z=w = 0}
= \begin{pmatrix} n\delta_{ij} \end{pmatrix}_{i, j=1}^m
\end{equation}

\begin{equation}\label{b2}B_n=\frac 1 {d_n}\begin{pmatrix}
 \frac{\d^3 \Pi_n(z,w)}{\d z\d^2 \bar w}|_{z=w}&\frac{\d \Pi_n(z,w)}{\d z}|_{z=w}\end{pmatrix}
= \begin{pmatrix} 0 & 0 \end{pmatrix}
\end{equation}
(an $m\times \frac {m^2+m+2}{2}$ matrix);

\begin{equation}\label{c2}C_n=\frac 1 {d_n}\begin{pmatrix}
\frac{\d^4 \Pi_n(z,w)}{\d^2 z\d^2\bar w}|_{z=w}&\frac{\d^2 \Pi_n(z,w)}{\d^2 z}|_{z=w}
\\\frac{\d^2 \Pi_n(z,w)}{\d^2 \bar w}|_{z=w}
&\Pi_n(z,z)
\end{pmatrix} = \begin{pmatrix} 2 n(n-1) I & 0 \\ &\\ 0 & 1
\end{pmatrix}
\end{equation}
(an $\frac {m^2+m+2}{2}\times \frac {m^2+m+2}{2}$ matrix.)

Since $B = 0$,
$$\begin{aligned} \Lambda_n&=\frac{n^{m}}{d_n}(C_n-B_n A_n^{-1}B^*_n) = \frac{n^m}{d^n} C_n\\& =\frac{n^{m}}{d_n} \begin{pmatrix}
 n(n-1) P  &0
\\ 0
&1
\end{pmatrix}
 \end{aligned}$$
where  \begin{equation}\label{p}P:=(\delta_{jj'}\delta_{qq'}+\delta_{jq'}\delta_{qj'})_{\frac{m^2+m}{2}\times \frac{m^2+m}{2}}\end{equation}
Thus \begin{equation} \label{CPL} \det \Lambda_n=(n^2)^{\frac {m^2+m}{2}}\det P,\end{equation}
and
\begin{equation}\label{denty2}p_z(y,0,\xi) =
\frac 1{\pi^{d_m} \det P}\frac 1{n^{m}n^{m^2+m}}e^{-(n(n-1))^{-1}P^{-1}|\xi|^2-|y|^2}
\end{equation}
Write $y$ in the polar coordinate $(x,\theta)$ and combine Lemmas \ref{KR} and \ref{pz}, we have
\begin{equation}\begin{aligned}\D^{SU(m+1)}_n(x)&=\frac 1{\pi^{d_m}}\frac {xe^{-x^2}}{n^{m^2+3m}}\int_0^{2\pi} \int_{M}\int_{S(\C^m)}\frac 1{ \det P}e^{-(n(n-1))^{-1}P^{-1}|\xi|^2}\\& \left|\det\left( |\xi|^2-n^2 x^2\right)\right|d\xi dV d\theta\end{aligned}\end{equation}

Now we change variables $\xi\to \sqrt{n(n-1) }\xi$, apply the assumption $\int_M dV=1$ and integrate $\theta$ variable to get,
\begin{equation}\label{SUM}\D^{SU(m+1)}_n(x)=\frac {2\pi xe^{-x^2}}{\pi^{d_m}} \int_{S(\C^m)}\frac 1{ \det P}e^{-P^{-1}|\xi|^2} \left|\det\left( \frac{n-1}{n}|\xi|^2-x^2\right)\right|d\xi\end{equation}
Now change variables $\xi\to \sqrt{P}\xi$ to obtain the stated result.

\end{proof}

For $m = 1$ we get
\begin{equation}\label{result2b}\begin{array}{lll} \D^{SU(2)}_n(x)&= & \frac 4{\pi}xe^{-x^2}\int_{0}^\infty
e^{-r} \left|2\frac{n-1}{n} r-x^2\right|dr\\&&\\  &= & \frac{n}{n-1}\frac 4{\pi}xe^{-x^2}\int_{0}^\infty
e^{-r} \left|2 r- \frac{n}{n-1} x^2\right|dr
\\ &&\\
 & = &  x\left(2  x^2-4 \frac{n}{n-1}+8\exp \{-\frac{n-1}{2 n} x^2\}\right) e^{-x^2}\end{array} \end{equation}

\section{Proof of Theorem \ref{maintheorem}}
In this section, we will complete the proof of Theorem \ref{maintheorem}. The main additional ingredient is
the asymptotic expansion of the Szeg\"o kernel. Otherwise the computation is similar to the case of $\CP^m$.

\subsection{Covariance matrix}\label{covariance}
We now calculate the leading terms in $A_n$ and $\Lambda_n$. The key
point is to calculate the mixed derivatives of $\Pi_n$ on the
diagonal. It is  convenient to do the calculation in K\"ahler
normal coordinates about a point $z_0$ in $M$.

Now take two derivatives on both sides of (\ref{bergman}),
\begin{equation}\begin{aligned}
\d_{j}\bar \d_{j'}\Pi_n(z, z)=& n^m (n\d_j\dbar_{j'}\phi e^{n\phi}[1+n^{-1}a_1+\cdots]+e^{n\phi}[n^{-1}\d_j\dbar_{j'} a_1+\cdots]\\
&+2\Re \d_j\phi e^{n\phi}[\dbar_{j'} a_1+\cdots]+n^2\d_j\phi\dbar_{j'}\phi e^{n\phi}[1+n^{-1} a_1+\cdots])
\end{aligned}\end{equation}

We apply (\ref{kahler}) at the origin $z_0$ to get,
\begin{equation}\label{a3}A_n(z_0)=\d_{j}\bar \d_{j'}\Pi_n(z_0, z_0)=\frac{n^{m}}{d_n}\left(nI+a_1 I+n^{-1}(a_2 I+\d_j\bar \d_{j'}a_1+\cdots)\right). \end{equation}

By the same arguments, we compute $B(z_0)$ up to the leading order term,

\begin{equation}B_n(z_0)=\d_{j}\bar \d_{j'}\bar\d_{q'}\Pi_n(z_0, z_0)=\frac{n^{m}}{d_n}[\delta_{jj'}\bar\d_{q'}a_1+\delta_{jq'}\bar\d_{j'}a_1+\cdots, \,\,\,n^{-1}\d_j a_1+\cdots]; \end{equation}
and the leading terms in each entry of $C_n(z_0)$,
\begin{equation}C_n(z_0)=\frac{n^{m}}{d_n}\begin{pmatrix}
  n^2(\delta_{jj'}\delta_{qq'}+\delta_{jq'}\delta_{qj'})(1+n^{-1}a_1+\cdots)&n^{-1}\d_j \d_q a_1+n^{-2}\d_j\d_q a_2+\cdots
\\ n^{-1}\bar\d_j\bar\d_q a_1+n^{-2}\bar\d_j\bar \d_q a_2+\cdots
&1+n^{-1}a_1+\cdots
\end{pmatrix}\end{equation}
We refer to \cite{DSZ2} for more details of those computations.

Together with Lemmas \ref{KR} and \ref{pz}, the formulae for $A_n, B_n, C_n$ give an explicit formula
for $\D_n$ on a general K\"ahler manifold.

\subsection{Proof of Theorem \ref{maintheorem}}
In this section we calculate the leading order term as $n \to \infty$ of the expected density $\D_n(x)$. We have,
$$\begin{aligned} \Lambda_n&=\frac{n^{m}}{d_n}(C_n-B_n A_n^{-1}B^*_n)\\& =\frac{n^{m}}{d_n} \begin{pmatrix}
 n^2 P+O(n)  &n^{-1}\d_j \d_q a_1+O(n^{-2})
\\ n^{-1}\bar\d_j\bar \d_q a_1+O(n^{-2})
&1+n^{-1}a_1+O(n^{-2})
\end{pmatrix}\\ &= \begin{pmatrix}
 n^2 P+O(n)  &n^{-1}\d_j \d_q a_1+O(n^{-2})
\\ n^{-1}\bar\d_j\bar \d_q a_1+O(n^{-2})
&1+n^{-1}a_1+O(n^{-2})
\end{pmatrix}
 \end{aligned}$$
in the last step, we apply the full expansion of $d_n$ (\ref{rock}) and where $P$ is given in \eqref{p}.
Thus \begin{equation}\det \Lambda_n\sim(n^2)^{\frac {m^2+m}{2}}\det P .\end{equation}
Applying (\ref{rock}) to (\ref{a}) again, we obtain the asymptotic analogoue of \eqref{CPL},  \begin{equation}\det A_n\sim n^{m}\end{equation}

We then have,  \begin{equation}\Lambda_n^{-1}\sim \begin{pmatrix}
 n^{-2} P^{-1}  &n^{-3}\d_j \d_q a_1
\\ n^{-3}\bar\d_j\bar \d_q a_1
&1
\end{pmatrix}+\mbox{lower order terms}\end{equation}
Thus,
\begin{equation}\label{denty2b}p_z(y,0,\xi) =\frac 1{\pi^{d_m} \det P}\frac 1{n^{m}n^{m^2+m}}e^{-n^{-2}P^{-1}|\xi|^2-|y|^2}+\mbox{lower order terms},
\end{equation}
hence by Lemma \ref{KR},
\begin{equation}\begin{aligned}\D_n(x)&=\frac 1{\pi^{d_m}}\frac {xe^{-x^2}}{n^{m^2+3m}}\int_0^{2\pi} \int_{M}\int_{S(\C^m)}\frac 1{ \det P}e^{-n^{-2}P^{-1}|\xi|^2}\\& \left|\det\left(|\xi|^2-n^2 x^2\right)\right|d\xi dV d\theta+\mbox{lower order terms}\end{aligned}\end{equation}
here we substitute $y=xe^{i\theta}$ in the expression of \eqref{denty2b}.

The remainder estimate is discussed in \S\ref{REMAINDER} below.

Now we change variables $\xi\to n\xi$, apply the assumption $\int_M dV=1$ and integrate in the $\theta$ variable to get,
\begin{equation}\D_n(x)=\frac {2\pi xe^{-x^2}}{\pi^{d_m}} \int_{S(\C^m)}\frac 1{ \det P}e^{-P^{-1}|\xi|^2} \left|\det\left(|\xi|^2-I x^2\right)\right|d\xi +O(\frac 1n )\end{equation}

Now change variable $\xi\to \sqrt{P}\xi$ to get,
\begin{equation}\begin{array}{lll} \label{DNDEF} \D_n(x) &= &c_mxe^{-x^2} \int_{S(\C^m)}e^{-|\xi|^2} \left|\det\left(|\sqrt P\xi|^2-I x^2\right)\right|d\xi+O(\frac{1}{n}) \\ &&\\
& = & f_m(x) + O(\frac{1}{n}). \end{array}\end{equation}

\subsection{\label{REMAINDER} Remainder estimate}

To prove the remainder estimate we use more terms in the Bergman kernel expansion and Taylor expansions of the various functions
in the Kac-Rice integrand. We sketch the proof as follows.  First we have,
$$\begin{aligned} \Lambda_n&=\frac{n^{m}}{d_n}(C_n-B_n A_n^{-1}B^*_n)\\& =\frac{n^{m}}{d_n} \begin{pmatrix}
 n^2 P(1+n^{-1}a_1+...)  &n^{-1}\d_j \d_q a_1+O(n^{-2})
\\ n^{-1}\bar\d_j\bar \d_q a_1+O(n^{-2})
&1+n^{-1}a_1+O(n^{-2})
\end{pmatrix}\\ &= \frac{1}{1+n^{-1}a_1+\cdots}\begin{pmatrix}
 n^2 P(1+n^{-1}a_1+\cdots)  &n^{-1}\d_j \d_q a_1+O(n^{-2})
\\ n^{-1}\bar\d_j\bar \d_q a_1+O(n^{-2})
&1+n^{-1}a_1+O(n^{-2})
\end{pmatrix}
 \end{aligned}$$

Hence there exists a complete asymptotic expansion, $$\det \Lambda_n=(n^2)^{\frac {m^2+m}2}\det P (1+n^{-1}b_1+n^{-1}b_2+\cdots)$$
Same apply to $$\det A_n=n^m(1+n^{-1}c_1+n^{-2}c_2+\cdots)$$
where $b_j$ and $c_j$ are polynomials of curvature and uniformly bounded.

This two identities imply the full expansion,
$$\det \Delta=\det A_n \det \Lambda_n=n^{m^2+2m}(1+n^{-1}d_1+n^{-2}d_2+\cdots)\det P$$

We also have,
\begin{equation}\Lambda_n^{-1}= \begin{pmatrix}
 n^{-2} P^{-1}(1-n^{-1}a_1+\cdots)  &n^{-3}(\d_j \d_q a_1+\cdots)
\\ n^{-3}(\bar\d_j\bar \d_q a_1+\cdots)
&1-n^{-1}a_1+\cdots.
\end{pmatrix} \end{equation}
It follows that
$$\begin{array}{lll} p_z^n(y,0,\xi)  &= &\frac 1{\pi^{d_m}(1+n^{-1}d_1+n^{-1}d_2+\cdots)\det P }\frac 1{n^{m}n^{m^2+m}}\\&&\\&&e^{-n^{-2}(1-n^{-1}a_1+\cdots)\xi^* P^{-1}\xi-(1-n^{-1}a_1+\cdots)|y|^2-2n^{-3}\Re [\langle(\d_j \d_q a_1+\cdots)\xi, y\rangle]}. \end{array}$$
where $\langle,\rangle$ is the nature inner product on $\C$.

We substitute this formula into the following integration in Lemma \ref{KR},
$$\D_n(x) = \frac 1{n^m\pi^{d_m}}\int_0^{2\pi}\int_M \int_{S(\C^m)}p^n_{z}(x,\theta,0,\xi)\left|\det\left(\xi\xi^*-n^2I x^2\right)\right|xd\xi dV d\theta, $$  then we change variables $\xi\to n\xi$ and $y=xe^{i\theta}$ to  obtain the exact formula,
$$\D_n(x) :=\frac{x}{\pi^{d_m}}\int_M \int_0^{2\pi}\int_{S(\C^m)} \frac{e^{- (1-n^{-1}a_1+\cdots)\xi^* P^{-1}\xi-(1-n^{-1}a_1+\cdots)x^2-2n^{-2}\Re [\langle\d_j \d_q a_1+\cdots)\xi, xe^{i\theta}\rangle]}}{(1+n^{-1}d_1+n^{-2}d_2+\cdots)\det P} $$$$ |\det(|\xi\xi^*-Ix^2)|)d\xi d\theta dV$$

Put $h=\frac 1 n$, then we rewrite the above formula $\D_x(h):=\D_n(x) $. $\D_x(h)$ is analytic with respect to $h$ for fixed $x$. We Taylor expand $\D_x(h)$ at $h=0$ to obtain,
$$\D_n(x) =D_\infty(x)+  O\left(\frac{1}{n} (x( 1 + x^4) e^{-x^2})\right).$$

\subsection{Riemann surface}

On a Riemann surface of area  $1$, we have
\begin{equation}\D_n(x)=\frac 2{\pi^2}xe^{-x^2}\int_{\C}e^{-|\xi|^2} \left|2|\xi|^2- x^2\right|d\xi +O(\frac 1n) ,\end{equation}
or equivalently \begin{equation}\label{result2}\D_n(x)=\frac 4{\pi}xe^{-x^2}\int_{0}^\infty e^{-r} \left|2r- x^2\right|dr +O(\frac 1n) \end{equation}

\section{Spherical ensemble: Proof of Theorem \ref{PBLIKE}}

In this section, we relate the expected density of critical points of an $L^2$ normalized  random $s_n \in SH^0(M, L^n)$ with
the spherical Haar measure to the expected density $\D_n =\D_n^{d_n}$ in the normalized Gaussian ensemble.
We begin by relating Gaussian and spherical averages of density-valued random variables in Lemma \ref{SPH}. This relation is valid both for the critical point distribution and the value distribution
of \S \ref{VDSec}.

\subsection{\label{RELATION} Relation between spherical and normalized Gaussian averages of density-valued random variables}

\begin{lem} \label{SPH}  The spherical density $\D_n^S$ and the $\gamma^n_{\alpha}$ densities of critical values are related by,
$$\begin{array}{lll} \D_n^{\alpha}(x)  &= &  \half  \frac{\alpha^{d_n}}{\pi^{d_n}}\omega_{2 d_n}\int_0^{\infty} \D_n^{S}(\rho^{-\half} x) e^{- \alpha \rho} \rho^{ d_n-\frac{3}{2}} d \rho \\&&\\
& = &  \half x^{2 d_n - 1}   \frac{\alpha^{d_n}}{ \pi^{d_n}}\omega_{2 d_n}\int_0^{\infty} \D_n^{S}(\rho^{-\half} ) e^{- \alpha x^2 \rho} \rho^{ d_n-\frac{3}{2}} d \rho.  \end{array}$$

\end{lem}

\begin{proof}

By \eqref{GSPH} and Lemma \ref{r}, we have

$$\begin{array}{lll} \E^n_{\alpha} CV_{s}  &= &   \frac{\alpha^{d_n}}{\pi^{d_n}}\int_{H^0(M, L^n)}  CV_{s}  e^{- \alpha  ||s||^2} D s  \\&&\\
&= &  \frac{\alpha^{d_n}}{\pi^{d_n}} \omega_{2 d_n} \int_0^{\infty} \int_{S H^0}    CV_{r s}  e^{- \alpha  r^2} r^{2d_n-1} dr d \nu_n
 \\ &&\\
&=&\frac{\alpha^{d_n}}{ \pi^{d_n}} \omega_{2 d_n} \int_0^{\infty} \int_{S H^0}    D_n^S(r^{-1} x) r^{-1} e^{- \alpha  r^2} r^{2d_n-1} dr d \nu_n
\end{array}$$
by definition of the normalization of the  Haar spherical
form on $SH^0(M, L^n)$).
We then change variables $\rho = r^2$.

\end{proof}

When $\alpha = d_n$, it follows from Lemma \ref{SPH} that
\begin{equation} \label{SPHDN} \begin{array}{lll}  \D_n^{d_n}( x)  &= &
 \half  \frac{\omega_{2d_n}}{\pi^{d_n}} d_n^{d_n} x^{2d_n-1}\int_0^{\infty} \D_n^{S}(\rho^{-\half} ) e^{- d_n x^2 \rho} \rho^{d_n-\frac{3}{2}} d \rho.
 \end{array} \end{equation}

Changing variables to $y = x^2$,  \eqref{SPHDN} is equivalent to
\begin{equation} \label{PBLAP1}
 \begin{array}{l} K_n y^{-d_n + \half}  d_n^{-d_n}\D_n^{d_n}(y^{\half}) = \lcal (\rho^{d_n-\frac{3}{2}}  d_S^n) (d_n y) ,
 \end{array}
\end{equation}
where $\lcal$ denotes the Laplace transform
and we put
\begin{equation} \label{AN} K_n = 2 \pi^{d_n} \omega_{2d_n}^{-1},\,\,\,d_S^n(\rho) = \D_n^S(\rho^{-\half}) .\end{equation}

If we  change variables $d_n \rho \to \rho$  in \eqref{SPHDN} we also get,

\begin{equation} \label{PBLAP2}
\begin{array}{l}   K_n y^{-d_n + \half} d_n^{-\half} \D_n^{d_n}(y^{\half}) = \lcal (\rho^{d_n-\frac{3}{2}}  D_{d_n^{-1}} d_S^n) ( y) .
 \end{array}
\end{equation}
In the last line,
$D_r f(x) =  f(r^{-1} x)$ denotes the  dilation operator.

\subsection{Spherical density for $SU(2)$ polynomials}

Before proving Theorem \ref{PBLIKE} it is helpful to do the calculations first for the case of $SU(m + 1)$
polynomials, where one has explicit formulae for $\D_n^{S, SU(m+1)}$. The formulae are  simplest when $m = 1$
and so we start with this case. We then follow the outline for the general case.

For $SU(2)$ polynomials, where $d_n = n + 1 ,$   \eqref{result2b} and \eqref{SPHDN}-\eqref{PBLAP2} imply that
\begin{equation} \label{NICE}
K_n d_n^{-\half} y^{-n}  \left(2  y-4 \frac{n}{n-1}+8\exp \{-\frac{n-1}{2 n} y\}\right) e^{-y
} = \lcal (\rho^{n-\frac{1}{2}}  D_{d_n^{-1}} d_S^n) ( y)   \end{equation}

We recall that
\begin{equation} \label{table} p^{-\nu} e^{-a p} =  \lcal {\bf 1}_{[a, \infty]} \frac{(x - a)^{\nu -1}}{\Gamma(\nu)}. \end{equation}
Hence the left side of \eqref{NICE} is $K_n d_n^{-\half} $ times
$$\lcal \left(2 {\bf 1}_{[1, \infty]}(\rho) [\frac{(\rho - 1)^{n-2}}{\Gamma(n -1)} - 4 \frac{n}{n-1}
\frac{(\rho -1)^{n-1} }{\Gamma(n)}] + 8 {\bf 1}_{[\frac{3n-1}{2n} , \infty]}(\rho)  \frac{(\rho - \frac{3n-1}{2n} )^{n-1})}{\Gamma(n)}\right).$$
 It follows that $ \D_n^{S, SU(2)}(\sqrt{d_n} \rho^{-\half})=D_{d_n^{-1}d_S^n}$ equals
$$ \begin{array}{l}
 K_n   \rho^{-n + \frac{1}{2}}   d_n^{-\half}  \left(2 {\bf 1}_{[1, \infty]}(\rho) [\frac{(\rho - 1)^{n-2}}{\Gamma(n -1)} - 4 \frac{n}{n-1}
\frac{(\rho -1)^{n-1} }{\Gamma(n)}] + 8 {\bf 1}_{[\frac{3n-1}{2n}, \infty]}(\rho)  \frac{(\rho - \frac{3n-1}{2n} )^{n-1})}{\Gamma(n)}\right) \\\\ =  K_n d_n^{-\half}  \rho^{-\half} \left(2
  {\bf 1}_{[1, \infty]}(\rho) [\rho^{-1} \frac{(1 - \rho^{-1})^{n-2}}{\Gamma(n -1)} - 4
 \frac{n}{n-1}
\frac{(1- \rho^{-1})^{n-1}}{\Gamma(n)}] + 8  {\bf 1}_{[0, \frac{3n-1}{2n}]}(\rho)  \frac{(1 - \frac{3n-1}{2n} \rho^{-1}  )^{n-1}  )}{\Gamma(n)}\right) . \end{array}$$

We then put $u = \sqrt{d_n} \rho^{-\half}$ to get
$$ \begin{array}{l} \D_n^{S, SU(2)}(u)
=  K_n  d_n^{-1} u  \left(2  {\bf 1}_{[0, \sqrt{n}]}(u) [u^2 d_n^{-1}\frac{(1 - \frac{u^2}{n})^{n-2}}{\Gamma(n -1)} - 4
\frac{n}{n-1}
\frac{(1- \frac{u^2}{n})^{n-1} }{\Gamma(n)}] + 8  {\bf 1}_{[0, b_n]}(u)  \frac{(1 -\frac{u^2}{b_n}  )^{n-1}  )}{\Gamma(n)}\right) . \end{array}$$
where $b_n = \frac{2n(n+1)}{3n -1}$. Recalling that $d_n = n + 1$, we conclude that
$$ \begin{array}{l} \D_n^{S, SU(2)}(u)
= \frac{K_n}{n+1} u  \left(2  {\bf 1}_{[0, \sqrt{n}]}(u) [u^2 (n+1)^{-1} \frac{(1 - \frac{u^2}{n})^{n-2}}{\Gamma(n -1)} - 4
\frac{n}{n-1}
\frac{(1- \frac{u^2}{n})^{n-1} }{\Gamma(n)}] + 8  {\bf 1}_{[0, b_n]}(u)  \frac{(1 -\frac{u^2}{b_n}  )^{n-1}  )}{\Gamma(n)}\right) . \end{array}$$
We further recall  from \eqref{GSPH}  that
$\omega_{2 (n + 1) }= 2 \frac{\pi^{n + 1}}{\Gamma(n + 1)}$  and combine $(n + 1)^{-2}  \Gamma(n -1)^{-1}
\simeq \Gamma(n + 1)^{-1}$ (resp.  $(n + 1)^{-1}  \Gamma(n)^{-1} \simeq \Gamma(n + 1)^{-1}$) to reach the final:
\begin{cor}\label{CPSPH} The spherical density for $SU(2)$ polynomials of degree $n$ is given exactly by
$$\begin{array}{l} \D_n^{S, SU(2)}=u \left(  2 u^2 {\bf 1}_{[0, \sqrt{n}]}(u) [(1- \frac{u^2} {n})^{n-2}- 4 \frac{n}{n-1}  {\bf 1}_{[0, \sqrt{n}]}(u) [(1 - \frac{u^2}{n})^{n-1}] + 8 {\bf 1}_{[0, b_n]}(u)
(1 - \frac{3n-1}{2n} \frac{u^2}{n})^{n-1} \right).\end{array}$$ \end{cor}

It follows that
$$\begin{array}{l} \lim_{n \to \infty} \D_n^{S, SU(2)}( u) =
u \left(  2 u^2 - 4 + 8 e^{-\half u^2} \right) e^{- u^2},\end{array}$$
proving Theorem \ref{PBLIKE} in the $SU(2)$ case.

\subsection{Spherical density for $SU(m + 1)$ polynomials}
We now extend the proof to $SU(m + 1)$ polynomials for general $m$. We do not simplify the integral for $f_m$
in that case but still may use the explicit $n$-dependence to confirm the main result.

Let  $\beta_m = \dim_{\C} S(\C^m)$, and let  $S_1(\C^m)$ denote the unit sphere in the space of complex symmetric matrices with respect to the
usual inner product $\tr A^* A$.
By Proposition \ref{DSUM}, and since the determinant is homogeneous of order $2 \beta_m = \dim_{\R} S(\C^m)$, we have
\begin{equation}\label{CPMD}  \begin{array}{lll} \D^{SU(m+1)}_n(y^{\half}) & = &c_m y^{\half} e^{-y } \int_{S(\C^m)}e^{-|\xi|^2} \left|\det\left(\frac{n-1}{n}|\sqrt P\xi|^2- y\right)\right|d\xi\wedge d\bar\xi \\ &&\\
& = & c_m y^{\half} e^{-y } \int_0^{\infty} \int_{S_1} e^{- r^2}
\left|\det\left(\frac{n-1}{n}|\sqrt Pr \omega|^2- y\right)\right| r^{2 \beta_m -1}   dr d \omega
 \\ &&\\
& = & \frac{c_m}{2}  y^{\half} e^{-y } \int_0^{\infty} \int_{S_1} e^{- \rho}
\left|\det\left(\frac{n-1}{n}|\sqrt P \omega|^2- y \rho^{-1}\right)\right| \rho^{3 \beta_m - 1}   d \rho d \omega
 \\ &&\\
& = & \frac{c_m}{2}  y^{\half}  y^{3 \beta_m } e^{-y } \int_0^{\infty} \int_{S_1} e^{- \rho y}
\left|\det\left(\frac{n-1}{n}|\sqrt P \omega|^2-  \rho^{-1}\right)\right| \rho^{3 \beta_m - 1}   d \rho d \omega
 \\ &&\\
& = & \frac{c_m}{2}  y^{\half}  y^{3\beta_m} e^{-y } \lcal (\rho^{3 \beta_m - 1} F_{n}^{\CP^m} )(y), \\ &&\\
\mbox{where} \;\; F_n^{\CP^m}(\rho): &= & \int_{S_1}
\left|\det\left(\frac{n-1}{n}|\sqrt P \omega|^2-  \rho^{-1}\right)\right| d \omega\end{array} \end{equation}

We further set $a_m = 3 \beta_m$ and recall \eqref{AN}. It follows from \eqref{PBLAP2}  that (with $d_S^n(\rho): = \D_n^{S,SU(m+1)}(\rho^{-\half}) ))$
\begin{equation} \label{FIRST}
K_n  \frac{c_m}{2} d_n^{-\half} y^{-d_n + 1} y^{a_m } e^{-y } \lcal (\rho^{a_m-1} F_{n}^{\CP^m} )(y)=
  \lcal (\rho^{d_n-\frac{3}{2}}  D_{d_n^{-1}} d_S^{n}) ( y)   \end{equation}
We now simplify the left side using further identities for the Laplace transform.
Denote the translation operator by
$\tau_a f(t) = f(t - a)$ and the Heaviside step function by  $H(t) = t_+^0.$  Also,
denote the  $\nu$-fold primitive (fractional integral)  by
$$I^{\nu} f(x) =\int_0^x \frac{(x - t)^{\nu-1}}{\Gamma(\nu)} f(t) d t. $$
We have (see e.g. \cite{W} Theorem 8.1),
\begin{equation} \label{IDENTITIES} \left\{ \begin{array}{l} \lcal (H(t - a) f(t - a)) = e^{- as} \lcal f (s)
 \\ \\
\lcal I^{\nu } f = s^{-\nu} \lcal f,
\end{array} \right. \end{equation}
We use the  identities to simplify the left side of \eqref{FIRST}:
\begin{equation} \label{LAPEXP}\begin{array}{lll}    y^{-d_n + a_m +1}
   \lcal (\rho^{a_m-1} F_{n}^{\CP^m}  )(y) e^{-y}  & = &   y^{-d_n + a_m + 1}
\lcal  \tau_1 (H(\rho)  \rho^{a_m -1}  F_{n}^{\CP^m} ) (y)\\&&\\
& = & \lcal (  I^{d_n - 1 - a_m}   \tau_1 (H(\rho)  \rho^{a_m-1}  F_{n}^{\CP^m} ) (y).    \end{array} \end{equation}
Combining \eqref{FIRST} and \eqref{LAPEXP} and  uniqueness of the Laplace transform gives
\begin{equation} \label{SECOND}
 K_n \frac{ c_m}{2}   I^{d_n - 2 - a_m}   d_n^{-\half} \tau_1 (H(\rho)  \rho^{a_m-1}  F_{n}^{\CP^m} )(\rho)
=  \rho^{d_n-\frac{3}{2}}  D_{d_n^{-1}} d^{n }_S (\rho) \end{equation}

We have,
$$\begin{array}{lll} I^{d_n - 1 - a_m}   \tau_1 (H(\rho)  \rho^{a_m -1}   F_n^{\CP^m}(\rho))(\rho)  &=  &
\int_0^{\rho} \frac{(\rho - t)^{d_n - 1 - a_m-1}}{\Gamma(d_n - 1 - a_m)}      (H(t -1)  (t-1)^{a_m-1}  F_{n}^{\CP^m}  (t-1) )d t\\&&\\
&=  &
\int_1^{\rho} \frac{(\rho - t)^{d_n  - a_m-2}}{\Gamma(d_n - 1 - a_m)}      ((t-1)^{a_m-1}  F_{n}^{\CP^m}  (t-1) )d t
\\&&\\
&=  &
\int_0^{\rho-1} \frac{(\rho - t-1)^{d_n  - a_m-2}}{\Gamma(d_n - 1 - a_m)}      (t^{a_m-1}  F_{n}^{\CP^m}  (t) )d t.

 \end{array} $$

Therefore,
$$\begin{array}{lll} \D_n^{S, SU(m+1)}(d_n^{\half} \rho^{-\half}) & = & K_n \frac{c_m}{2}  \rho^{-d_n + \frac{3}{2}}  d_n^{-\half}
\int_0^{\rho-1} \frac{(\rho - t-1)^{d_n  - a_m-2}}{\Gamma(d_n - 1 - a_m)}
 (t^{a_m-1}  F_{n}^{\CP^m}  (t) )
dt\\&&\\
&=& K_n \frac{c_m}{2}  \frac{1}{\Gamma(d_n - 1 - a_m) }\rho^{ - a_m -  \frac{1}{2}}
 \int_0^{\rho-1} (1 - \frac{ t+1}{\rho})^{d_n  - a_m-2}
(t^{a_m-1}  F_{n}^{\CP^m}  (t) )  dt,
\end{array} $$
or equivalently with $u = d_n^{\half} \rho^{-\half}, \rho = u^{-2} d_n$,
$$\begin{array}{lll} \D_n^{S, SU(m+1)}(u) & = &
 K_n \frac{c_m}{2}  \frac{1}{ \Gamma(d_n - 1 - a_m) } d_n^{-\half}  (d_n u^{-2})^{ - a_m -  \frac{1}{2}}
 \\&&\\&& \int_0^{d_n u^{-2}-1} (1 - \frac{ (t+1)u^2}{d_n})^{d_n  - a_m-2}    (t^{a_m-1}  F_{n}^{\CP^m}  (t) ) dt
\end{array} $$

Substituting the definition of $ F_n^{\CP^m}(t)$, we observe that
\begin{equation} \label{LIMIT1} \begin{array}{l}
\lim_{n \to \infty}  \int_{S_1} \int_0^{d_n u^{-2}-1} (1 - \frac{ (t+1)u^2}{d_n})^{d_n  - a_m-2}    t^{a_m-1}
\left|\det\left(\frac{n-1}{n}|\sqrt P \omega|^2-  t^{-1}\right)\right| d \omega
\\ \\ =  e^{- u^2} \int_0^{\infty}  e^{- t u^2}   t^{a_m-1}  \left(\int_{S_1}
\left|\det\left(|\sqrt P \omega|^2-  t^{-1}\right)\right| d \omega\right) dt. \end{array} \end{equation}
Also,
$$ \Gamma(d_n - 1 - a_m) d_n^{ a_m + 1} = \frac{d_n^{ a_m + 1} }{(d_n-1)
\cdots (d_n - 1 - a_m) } \Gamma(d_n ) \simeq  \Gamma(d_n), $$
hence
$$\frac{K_n}{\Gamma(d_n - 1 - a_m) }  d_n^{- a_m -1}
\simeq  1$$
Reversing the steps in  \eqref{CPMD} and comparing with \eqref{DNDEF},  it follows that as $n \to \infty$,
\begin{equation} \label{LIMIT2}\begin{array}{lll}
\D_n^{S, SU(m+1)}(u)  &\simeq & \frac{c_m}{2}
 u^{2 a_m + 1} e^{- u^2} \int_0^{\infty}  e^{- \rho u^2}   \rho^{a_m-1}  \left(\int_{S_1}
\left|\det\left(|\sqrt P \omega|^2-  \rho^{-1}\right)\right| d \omega\right) d \rho\\&&\\
&\simeq &  D_{\infty}(u). \end{array} \end{equation}

\subsection{Proof of Theorem \ref{PBLIKE}}

We now go through the general case, closely following the calculations in the special case of $\CP^m$.

By  Theorem \ref{maintheorem},
$$\D_n^{d_n}(y^{\half})  = \D_{\infty}(y^{\half}) + O( \frac{y^k e^{- y}}{n}), \;\; \mbox{with}\;\;
\D_{\infty}(y^{\half}) = f_m(y^{\half}) y^{\half} e^{-y}. $$
In determining the limit of $D_n^S$ it follows from
  \eqref{PBLAP2}  and \eqref{SPHDN} and  the remainder estimate that we may drop the remainder term and then the calculation becomes
almost identical to that of $SU(m+1)$ polynomials, with $ F_n^{\CP^m}(\rho)$ replaced by
$$F_m(\rho) = \int_{S_1(\C^m)} \left|| \sqrt P \omega|^2- \rho^{-1}\right| d \omega. $$

As in \eqref{CPMD},  we have
\begin{equation} \label{f} f_m(\sqrt{y})
 = \frac{c_m}{2} y^{a_m + \half} e^{-y} \lcal (\rho^{a_m-1} F_m)(y). \end{equation}

By \eqref{PBLAP2}\eqref{f} (with $d_S^n(\rho) = \D_n^S(\rho^{-\half}) ))$ we have
\begin{equation} \label{final}  \begin{array}{l}
K_n \frac{c_m}{2}   d_n^{ \half}
y^{1 + a_m -d_n } (\lcal \rho^{a_m-1} F_m)(y)  e^{-y}= \lcal (\rho^{d_n-\frac{3}{2}} D_{d_n^{-1}} d_S^n)) ( y).
\end{array}  \end{equation}

 Hence
\begin{equation}\label{final2} \frac{c_m}{2}  K_n   I^{d_n - 2 - a_m}\tau_1 H(\rho) \rho^{a_m} F_m(\rho)  =
\rho^{d_n-\frac{1}{2}} D_{d_n^{-1} } d_S^n (\rho). \end{equation}

Repeating the calculation in the $SU(m)$ case  and setting  $u = \sqrt{d_n} \rho^{-\half}$ gives
\begin{equation} \label{dsn} \begin{array}{lll}
\D_n^S(u)&= & (d_n u^{-2})^{\frac{3}{2} - d_n}
\left(  \frac{c_m}{2} K_n \right)
\int_0^{d_n u^{-2} -1} \frac{(d_n u^{-2} -1- t)_+^{d_n - 2 - a_m}}{\Gamma(d_n - 1 - a_m)} t^{a_m-1} F_m(t) d t
\\&&\\
& = & \frac{c_m}{2}   K_n   \frac{ d_n^{- a_m - 2 } }{\Gamma(d_n - 1 - a_n)} \; u^{2a_m +2}
\int_0^{d_n u^{-2}-1}  (1- \frac{(1+ t))u^2}{d_n})_+^{d_n   - a_m-2} t^{a_m -1} F_m(t) d t.  \end{array}\end{equation}
We note that
$${\bf 1}_{[0, d_n u^{-2} -1]} (1- \frac{(1+ t))u^2}{d_n})_+^{d_n   - a_m-2} t^{a_m-1} F_m(t)
\to \exp (- (1 + t) u^2) t^{a_m-1} F_m(t) $$
montonically as $n \to \infty$ and so
$$\lim_{n \to \infty}
\int_0^{d_n u^{-2}-1}  (1- \frac{(1+ t))u^2}{d_n})_+^{d_n   - a_m-2} t^{a_m-1} F_m(t) d t
= e^{-u^2}\int_0^{\infty}   \exp (- \rho u^2)  \rho^{a_m-1} F_m(\rho)  dt. $$

The rest of the calculation is the same as for $\CP^m$ and we get
\begin{equation} \label{dsnb} \begin{array}{lll}
\lim_{n \to \infty} \D_n^S(u)&= & D_{\infty}(u). \end{array} \end{equation}

\section{\label{VDSec} Value distribution}

As a check on the results for the critical value distribution, we give the analogous calculation in this section of
the much simpler expected value distribution, which  is  a probability measure on $\R_+$.
By the value distribution of (the modulus of) a section $s_n\in H^0(M,L^n)$ we mean the probability measure $\mu_{s_n}$  on $\R_+$ defined
on a test function $\psi \in C(\R_+)$ by
\begin{equation} \label{VD} \mu_{s_n} =  (|s_n|_{h^n})_* dV, \;\; \langle \psi,  (|s_n|_{h^n})_* dV \rangle = \int_{M}  \psi (|s_n|_{h^n}) dV. \end{equation}
As is well-known, the distribution is minus  the derivative of  the volume function,
\begin{equation} \label{VOLFUNS} V_{s_n}(\lambda) : =  Vol\{z \in M:  |s_n(z)|_{h^n} > \lambda\}. \end{equation}

Fix $z \in M$ and define the random variable  $\rho^z_n(x)= |f_n(z)| e^{- n \phi(z)/2} = |s_n(z)|_{h^n} $. For a test function $\psi \in C(\R_+)$ we have (for any probability measure on sections)
\begin{align*}&\E_n\psi(\rho_n^z) = \int_0^{\infty} \psi(x) f_n^z(x) dx
\end{align*}
where $f_n^z(x)$ is the pointise density of the distribution of $\rho_n^z$.

By the expected value distribution in the Gaussian ensemble we mean the measures
\begin{equation} \mu_n:= \E_n \mu_{s_n}\end{equation} which is defined in the sense of distribution, \begin{equation} \langle \psi,   \E_n (|s_n|_{h^n})_* dV \rangle = \int_{M} \E_n \psi (|s_n(z)|_{h^n}) dV=\int_M \E_n\psi(\rho_n^z)dV. \end{equation}
To determine the expected value distribution it thus suffices determine $\E_n \psi (|s_n(z)|_{h^n}^2)$ or equivalently
$\E_n V_{s_n}(\lambda)$. We note that $\mu_n$ is always an absolutly continuous measure on $\R_+$ since $s_n$
is analytic.  Hence  we may write it as $f_n(x) dx$.

We also define the expected value distribution in the spherical ensemble ,
\begin{equation} \mu_n^S = \E_n^S \mu_{s_n}, \;\;\;\; s_n \in SH^0(M, L^n). \end{equation}
It is also a continuous measure and we write   $\mu_n^S = f^S_n(x) dx$.

The random variable $\rho^z_n$ is the modulus of the Gaussian random variable
$$p^z_n=f_n(z) e^{- n \phi(z)/2}$$
which is only defined with a choice of local coordinates. But if $\psi$ is a radial test function on $\C$, then
$$\E_n \psi(p^z_n) = \E_n \psi(\rho^z_n), $$
and the right side is globally well defined.
We denote the expected density of $p_n^z$ by $\widehat f_n^z(y)$, which is a density defined on $\C$.  We write it as 
$\widehat f_n^z(x,\theta)$ in polar coordinate $y=xe^{i\theta}$.

\begin{lem} \label{KRval} For any probability measure on sections, the expected density of  $f_n^z$ is given by the formula,
\begin{equation}f^z_n(x) = \int_0^{2\pi} \widehat f_n^z(x,\theta)xd\theta\end{equation}

\end{lem}


\subsection{Kac-Rice formula for the Gaussian value density}

First we  compare Gaussian value densities as $\alpha$ changes. We denote $f_n^\alpha$ is the expected density under the Gaussian ensemble \eqref{HGMg}.  The result is analogous to that of Lemma \ref{alpha}:

\begin{lem}\label{alphaf}
$f_n^{\alpha}(x) =  \alpha^{\half}f^1_n(\alpha^{\half} x) $.
\end{lem}

Hence it suffices to fix $\alpha = 1$, and in the next Proposition  we determine  the  Gaussian  value distribution when $\alpha=1$ in \eqref{HGMg} \eqref{law}. If we define the expected density of values
,\begin{equation}\E^1_n\left( \int_M\psi(|s_n|_{h^n})dV\right)=\int_0^\infty \psi(x)f_n^1(x)dx. \end{equation}
Then we can prove,

\begin{prop} \label{KRvalprop}The expected value density in the $\alpha =1$ ensemble is given by the formula,
\begin{equation}f_n^1(x) =\frac{1}{\pi} \int_M \frac 1{\Pi_n(z,z)}   x e^{- \Pi_n(z,z)^{-1} x^2} dV(z),\,\,\,x\in \R_+ \end{equation}

\end{prop}

\begin{proof}

It follows from the Kac-Rice formula  that
\begin{equation}\label{dddd} \widehat f_n^z(y) =\frac 1{\pi}\frac 1{\Pi_n(z,z)}\exp\left\{-\left\langle y,  \Pi_n(z,z)^{-1} \bar y \right \rangle\right\},
\end{equation}
since  the covariance matrix of the random variable $p_n^z$ is
$$\E_n^1 (f_n^z \overline{f_n^z} )= \E_n^1 |s(z)|_{h_n}^2 = \Pi_n(z,z).$$
To obtain the density of values, we apply Lemma \ref{KRval} and the polar coordinate
$y=xe^{i\theta}$ to \eqref{dddd}, then integrate the pointwise densities over $M$.

\end{proof}

The density is especially simple for $SU(m+1)$ polynomials where $\Pi_n(z,z)$ is the constant $\frac{d_n}{V_m}$
where $V_m$ is the Fubini-Study volume of $\CP^m$.

\begin{cor} \label{CPexact} For $SU(m+1)$ polynomials with $\alpha = 1$ resp. $ d_n$ and with volume $V_m$,
the expected density of critical values is given by

\begin{equation} \label{fnCPm}  f_n^{1, \CP^m} (x) =  \frac{ 2V_m}{d_n }  x e^{- V_m d_n^{-1} x^2},\;\;\;
f_n^{d_n, \CP^m} = 2V_m  x e^{- V_m x^2}.\end{equation}
\end{cor}
As mentioned in the introduction, we usually set $V_m = 1$ for simplicity of notation.

\subsection{Relation between spherical and Gaussian densities}

  The Gaussian value distribution has the same problem as the critical point distribution, namely the weighted
repetition of sections. If we multiply $s \in H^0(M, L^n)$ by $r > 0$ then the volume function changes by
$V_{rs}(\lambda) = V_s(r^{-1} \lambda)$. Consequently,
\begin{equation} \label{frs} f_{rs} = r^{-1} f_s(r^{-1} s). \end{equation}

The discussion is a word-for-word repetition
of \S \ref{RELATION}.  Indeed, in the calculations we only used the relation \eqref{frs} and the fact that the
random variables take values in the densities on $\R$. We therefore omit the proofs, but for clarity do state
the analogous Lemmas.

Exactly
as in Lemma \ref{SPH}, we have

\begin{lem} \label{SPHf}  The spherical density $f_n^S$ and the $\gamma^n_{\alpha}$ densities of values are related by,
$$\begin{array}{l}  f_n^{\alpha}(x) =  K_n^{-1} \alpha^{d_n} \int_0^{\infty} f_n^{S}(\rho^{-\half} x) e^{- \alpha \rho} \rho^{-\half} \rho^{d_n-1} d \rho.  \end{array}$$
\end{lem}

We combine Lemmas \ref{alphaf} and \ref{SPHf} to obtain

\begin{cor} \label{CORf}
$$\begin{array}{lll}  f^1_n(\alpha^{\half} x) &= &  \alpha^{-\half}   \alpha^{d_n}  K_n^{-1} \int_0^{\infty} f_n^{S}(\rho^{-\half} x)  e^{- \alpha \rho} \rho^{d_n -\frac{3}{2}} d \rho \\&&\\ &=&      \alpha^{d_n -\half} K_n^{-1}   x^{2d_n -1}  \int_0^{\infty} f_n^{S}(\rho^{-\half} ) e^{- \alpha x^2 \rho} \rho^{d_n-\frac{3}{2}} d \rho. \end{array}$$

\end{cor}

\subsection{Spherical density of values: Proof of Theorem \ref{SDV} }

In this section we prove Theorem \ref{SDV}.
First we prove the statement for $SU(m+1)$ polynomials.

\begin{lem} Let $f_n^{S,\CP^m}(u)$ be the density of values of $SU(m+1)$ polynomials in the spherical ensemble
with volume $V_m = 1$.
Then
$$\lim_{n \to \infty} f^{S,\CP^m}_n(u) = 2 u e^{- u^2}. $$
\end{lem}

\begin{proof}
The  spherical density of values is determined from the Corollaries  \ref{CPexact} and \ref{CORf} (with $V_m=1$) by changing variable
$y = \alpha x^2$,
$$ 2 K_n \frac{1}{ d_n}  y^{-d_n + 1} e^{-  d_n^{-1} y} =  \lcal (\rho^{d_n-\frac{3}{2}} F^S_n) (y). $$
where we put
$$F^S_n(\rho) = f^{S,\CP^m}_n(\rho^{-\half}). $$

By \eqref{table}
with   $a =  d_n^{-1}$ and $\nu =  d_n - 1$, we obtain
$$\lcal (\rho^{d_n-\frac{3}{2}} F^S_n) (y) = 2K_n  \frac {1}{d_n}  \lcal {\bf 1}_{[d_n^{-1}, \infty]}
\frac{(\rho -  d_n^{-1})^{d_n-2}}{\Gamma(d_n -1)}. $$
Thus,
$$ f^{S,\CP^m}_n (\rho^{-\half}) = 2 K_n \rho^{-d_n+\frac{3}{2}}  \frac {1}{ d_n}   {\bf 1}_{[d_n^{-1}, \infty]}
\frac{ (\rho -  d_n^{-1})^{d_n-2}}{\Gamma(d_n-1)}. $$
 With $u = \rho^{-\half}$ we get

\begin{equation} \label{fsn} \begin{array}{lll}  f^{S,\CP^m}_n (u) & = & u^{2d_n-3}  \frac{2K_n }{\Gamma(d_n-1)} \frac {1}{ d_n}   {\bf 1}_{[d_n^{-1}, \infty]} (u^{-2} - d_n^{-1})^{d_n-2}\\&&\\
& = &2\frac{K_n }{\Gamma(d_n)}    u  {\bf 1}_{[0, \sqrt{d_n}]]} (1- \frac{ u^2}{d_n})^{d_n-2} .
\end{array} \end{equation} The Lemma follows.

\end{proof}

\subsection{General K\"ahler manifolds}

In the general case, we also have:

\begin{lem} Let $f_n^{S}(u)$ be the density of values of random $s \in SH^0(M, L^n)$  in the spherical ensemble.
Then
$$\lim_{n \to \infty} f^{S}_n(u) = 2 u e^{- u^2}. $$
\end{lem}

\begin{proof} We combine Proposition \ref{KRvalprop} and Corollary \ref{CORf}. The exact formula of the Proposition
shows that the general spherical density is asymptotic to the $SU(m + 1)$ case in dimenson $m$ with a remainder
term of order $n^{-1} (1 + |x|^4)e^{-|x|^2}$. The remainder estimate is similar to but simpler than that for the
critical point density in \S \ref{REMAINDER} and is omitted.  It follows that
$$ f^S_n (\rho^{-\half}) \simeq  \rho^{-d_n+\frac{3}{2}}  2 K_n \frac {1}{ d_n}   {\bf 1}_{[d_n^{-1}, \infty]}
\frac{ (\rho -  d_n^{-1})^{d_n-2}}{\Gamma(d_n-1)}. $$ The rest of the calculation is then identical to the
case of $SU(m + 1)$ polynomials.

\end{proof}


\begin{thebibliography}{HHHH}

\bibitem[ABA]{ABA} A. Auffinger and  G.  Ben Arous,
Complexity of random smooth functions of many variables,
arXiv:1110.5872.

\bibitem[AT]{AT} R. Adler, J. Taylor,\emph{ Random fields and geometry}. Springer Monographs in Mathematics Springer, New York (2007).


\bibitem[B]{B}  J. Baber,
Scaled Correlations of Critical Points of Random Sections on Riemann Surfaces, J. Stat. Phys.  148 (2012), 250--279.


\bibitem[Bau]{Bau}  B.  Baugher, Asymptotics and dimensional dependence of the number of critical points of random holomorphic sections. Comm. Math. Phys. 282 (2008), no. 2, 419--433; Metric dependence and asymptotic minimization of the expected number of critical points of random holomorphic sections. Trans. Amer. Math. Soc. 362 (2010), no. 9, 4537--4555.

\bibitem[BSZ]{BSZ} P. Bleher, B. Shiffman and S. Zelditch,  Universality
and scaling of correlations between zeros on complex manifolds,
Invent. Math. 142 (2000), 351--395;  Universality and scaling of zeros on symplectic manifolds. Random matrix models and their applications, 31--69, Math. Sci. Res. Inst. Publ., 40, Cambridge Univ. Press, Cambridge, 2001.




\bibitem[DSZ1]{DSZ1} M. R. Douglas, B. Shiffman and S. Zelditch, Critical Points and Supersymmetric Vacua I, Commun. Math. Phys. 252, 325--358 (2004).
     \bibitem[DSZ2]{DSZ2}   M. R. Douglas, B. Shiffman and S. Zelditch, Critical Points and Supersymmetric Vacua II: asymptotics and extremal metrics, J. Differential. Geom.
72, (2006), 381--427.
\bibitem[DSZ3]{DSZ3}M. R. Douglas, B. Shiffman and S. Zelditch, Critical Points and Supersymmetric Vacua III, Commun. Math. Phys. 265, 617--671 (2006).

 \bibitem[FW]{FW}R. Feng and Z. Wang, Critical values of Gaussian SU(2) random polynomials,  arXiv:1210.4829.


\bibitem[GW]{GW}  D. Gayet and J. Y.  Welschinger,
What is the total Betti number of a random real hypersurface?, to appear in Crelle J.
arXiv:1107.2288.


\bibitem[GH]{GH}G. Griffiths and J. Harris,\emph{ Principles of Algebraic Geometry}. Wiley-Interscience, (1978).

\bibitem[H]{H} B. Hanin,
Correlations and Pairing Between Zeros and Critical Points of Gaussian Random Polynomials, arXiv:1207.4734.

\bibitem[K]{K} M. Kac, The average number of real roots of a random algebraic equation, Bull. A.M.S. 49(1943), 314--320.

\bibitem[L]{L}Z. Lu,  On the Lower Order Terms of the Asymptotic Expansion of Zelditch, Amer. J. Math, vol 122(2), 2000, pp 235--273.
\bibitem[Mac]{Mac} B. Macdonald,  Density of complex critical points of a real random SO(m+1) polynomial. J. Stat. Phys. 141 (2010), no. 3, 517--531.

\bibitem[NSV]{NSV} F. Nazarov, M. Sodin and A. Volberg,  Transportation to random zeroes by the gradient flow, Geom. Funct. Anal. 17 (2007), no. 3, 887--935.


\bibitem[N1]{N1}   L. Nicolaescu, Critical sets of random smooth functions on compact manifolds I,  arXiv:1008.5085;
Critical sets of random smooth functions on compact manifolds II,   arXiv:1101.5990.

\bibitem[N2]{N2}  L. Nicolaescu,
Complexity of random smooth functions on compact manifolds. I,
arXiv:1201.4972.




\bibitem[R]{R}  S. O. Rice, Mathematical analysis of random noise, Bell System Tech. J. 23 (1944), 282--332.

\bibitem[SZ]{SZ}  B. Shiffman and S.  Zelditch,
Random polynomials of high degree and Levy concentration of measure,
Asian J. Math. 7 (2003), no. 4, 627--646.

\bibitem[SZ2]{SZ2}  B. Shiffman and S.  Zelditch,  Addendum: "Asymptotics of almost holomorphic sections of ample line bundles on symplectic manifolds'' [J. Reine Angew. Math. 544 (2002), 181--222;  Proc. Amer. Math. Soc. 131 (2003), no. 1, 291--302.

\bibitem[SZ3]{SZ3} B. Shiffman and S. Zelditch, Distribution of critical values of Gaussian random holomorphic fields,
(unpublished manuscript, 2009).

\bibitem[T]{T} G. Tian, On a set of polarized \kahler metrics on algebraic manifolds, J. Diff. Geom. 32 (1990), 99--130.

\bibitem[W]{W}  D. V. Widder, {\it The Laplace Transform. } Princeton Mathematical Series, v. 6. Princeton University Press, Princeton, N. J., 1941.


\bibitem[Z]{Z}S. Zelditch, Real and complex zeros of Riemannian random waves, Spectral analysis in geometry and number theory, 321--342, Contemp. Math., 484, Amer. Math.Soc., Providence, RI, 2009.

 \bibitem[Ze]{Ze}S. Zelditch, Szeg\"{o} kernels and a theorem of Tian, IMRN 6 (1998), 317--331.
\end{thebibliography}
\end{document}